\documentclass[twoside,11pt]{article}

\usepackage{jmlr2e}
\usepackage{amsmath}
\usepackage{amssymb}
\usepackage{mathtools}
\usepackage{xcolor}
\newtheorem{assumption}{Assumption}

\DeclarePairedDelimiterX{\norm}[1]{\lVert}{\rVert}{#1}
\DeclarePairedDelimiterX{\bignorm}[1]{\big\lVert}{\big\rVert}{#1}

\DeclareMathOperator*{\Mat}{\text{Mat}}
\DeclareMathOperator*{\bbR}{\mathbb{R}}
\DeclareMathOperator*{\bbN}{\mathbb{N}}
\newcommand{\X}{\mathcal{X}}
\newcommand{\calH}{\mathcal{H}}
\newcommand{\bbE}{\mathbb{E}}

\def\fGP{f_{\text{GP}}}
\DeclareMathOperator*{\argmax}{arg\,max}
\DeclareMathOperator*{\argmin}{arg\,min}
\newcommand\restr[2]{{
  \left.\kern-\nulldelimiterspace 
  #1 
  \right|_{#2} 
  }}

\jmlrheading{0}{0000}{0-0}{00/00}{00/00}{meila00a}{George Wynne, Fran\c{c}ois-Xavier Briol and Mark Girolami}

\ShortHeadings{Convergence Guarantees for Gaussian Process Means}{Wynne, Briol and Girolami}
\firstpageno{1}

\begin{document}

\title{Convergence Guarantees for Gaussian Process Means \\ With  Misspecified Likelihoods and Smoothness}

\author{\name George Wynne \email g.wynne18@imperial.ac.uk \\
       \addr Department of Mathematics\\
       Imperial College London\\
       London, SW7 2AZ, UK\\
       \AND
       \name Fran\c{c}ois-Xavier Briol \email f.briol@ucl.ac.uk \\
       \addr Department of Statistical Science\\
       University College London\\
       London, WC1E 7HB, UK\\
       \AND
       \name Mark Girolami \email mag92@eng.cam.ac.uk \\
       \addr Department of Engineering \\
       University of Cambridge\\
       Cambridge, CB2 1PZ, UK\\
       }

\editor{}

\maketitle

\begin{abstract}
Gaussian processes are ubiquitous in machine learning, statistics, and applied mathematics. They provide a flexible modelling framework for approximating functions, whilst simultaneously quantifying uncertainty. However, this is only true when the model is well-specified, which is often not the case in practice. In this paper, we study the properties of Gaussian process means when the smoothness of the model and the likelihood function are misspecified. In this setting, an important theoretical question of practical relevance is how accurate the Gaussian process approximations will be given the chosen model and the extent of the misspecification. The answer to this problem is particularly useful since it can inform our choice of model and experimental design. In particular, we describe how the experimental design and choice of kernel and kernel hyperparameters can be adapted to alleviate model misspecification.
\end{abstract}



\section{Introduction} \label{sec:introduction}

Gaussian processes (GPs) have found widespread use in machine learning \citep{Rasmussen2006} as they offer flexible and interpretable models with uncertainty quantification.  Applications include reinforcement learning \citep{Kuss2004}, time-series modelling \citep{Roberts2013}, robotics and control \citep{Deisenroth2015}, as well as Bayesian numerical methods including Bayesian quadrature \citep{Briol2019,Kanagawa2019}, Bayesian optimization \citep{Mockus1989,Snoek2012,Bull2011} and Bayesian differential equations solvers \citep{Cockayne2016}. Outside of machine learning, Gaussian process regression was first used in geostatistics \citep{Krige1951,Cressie1990,Matheron1963}, where the procedure was originally known as kriging and is a current active field \citep{Wang2019,Lederer2019}. Gaussian processes are used for tackling problems ranging from computer models \citep{Kennedy2001} to inverse problems \citep{Stuart2010,Stuart2018}, health monitoring \citep{Stegle2008}, engineering design \citep{Forrester2008} and tsunami modelling \citep{Sarri2012}, to name but a few.

In most of the applications above, the central task is to approximate a function of interest given pointwise evaluations of this function which may be corrupted by some unknown noise. To do so, practitioners carefully design their algorithms such that the approximation error decreases at a fast rate in the number of data points. Several modelling choices need to be made, including the selection of a GP model and hyperparameters, of a likelihood, and of the locations at which to obtain data. Making appropriate choices for a given application is an extremely difficult task, and poor choices can lead to poor empirical performance. One way to tackle this problem in a unified manner is to turn to theoretical convergence guarantees which explicitly account for these modelling choices, and to select specific algorithms which minimise upper bounds on the approximation error. 

Of course, this approach is only sensible if the bounds apply to the problem at hand, but most existing bounds are rather restrictive and require assumptions which users might not be able to verify. The novel contributions of this paper include convergence guarantees in the presence of two common modelling errors, and suggestions as to how to construct algorithms which can mitigate these. 

The first is \emph{likelihood misspecification}, meaning that the observations follow a distribution which is different from the one assumed by the model. This often occurs because conditioning of Gaussian process means on data is only possible in closed-form if assuming the data is noiseless, or contains independently and identically distributed Gaussian noise with known variance. For more complex observations, such as input-dependent noise \citep{Goldberg1998,Le2005} or distributions with heavy tails \citep{Vanhatalo2009}, a closed-form expression for the mean is not available. In order to maintain a closed-form expression, practitioners often use simplistic models which may not be a faithful representation of the data-generating process, leading to a lack of robustness and poor approximations \citep{Goldberg1998,Jylanki2011}.

The second is \emph{smoothness misspecification}, meaning that the Gaussian process mean is either too rough or too smooth relative to the target function. Here, the smoothness of a function is measured in terms of number of derivatives in the sense of Sobolev spaces. This is known to guide the rate of convergence of Gaussian process approximations, with faster rates attainable for smoother functions if the mean and covariance functions are chosen appropriately. However, for many of the aforementioned applications, it is difficult to identify the smoothness of the target function. This commonly leads to sub-optimal choices of GPs, and as a result potentially slower convergence rates. 

Our novel convergence guarantees highlight the impact that both types of misspecification can have on rates of convergence, and can provide guidance on model choice for practitioners at risk of misspecification. In particular, the impact of the experimental design and covariance function is made clear in the bounds. The bounds employ results from the scattered data approximation (SDA) literature \citep{Wendland2005}, which has been applied to GP related methods in numerous works \citep{Bull2011,Stuart2018,Xi2018MultiOutput,Briol2019,Teckentrup2019,Tuo2019}. Smoothness misspecification has previously been considered in this context \citep{Narcowich2006,Teckentrup2019,Kanagawa2019} as has corrupted observations \citep{Rieger.Zwicknagl:2009,Arcangeli2007}. However, the interplay of smoothness and likelihood misspecification has not been investigated to date. Our paper therefore unifies and extends existing work in this area.

The main results in this paper are Theorem \ref{thm:Misspecified_Sobolev_Interpolation}, Theorem \ref{thm:Misspecified_Sobolev_Random_Regression}, Theorem \ref{thm:Misspecified_Sobolev_Det_Regression} and Theorem \ref{thm:misspec_likelihood_inter} which, respectively, concern the cases when a likelihood reflecting no noise is correctly assumed, a Gaussian likelihood is correctly assumed, a Gaussian likelihood is incorrectly assumed and a likelihood of no noise is assumed but there is arbitrary corruption. In each case the results also facilitate the smoothness of the target function being different from the smoothness of the approximating function. To highlight the relevance of these novel bounds, in Section \ref{sec:bayesian_pn} we derive implications for the convergence of Bayesian numerical methods based on GPs, specifically Bayesian quadrature and Bayesian optimization.

The paper is structured as follows. Section \ref{sec:background} reviews background material on GPs and reproducing kernel Hilbert spaces. Section \ref{sec:assumptions} introduces and discusses assumptions on the design region, design points and GP model required for our theory to hold. Existing convergence results are also covered. Section \ref{sec:theory} contains the error bounds. Section \ref{sec:bayesian_pn} demonstrates implications of these bounds for Bayesian quadrature and Bayesian optimization. Section \ref{sec:conclusions} provides concluding remarks.

\section{Background on Gaussian Processes and Kernel Methods}\label{sec:background}

In this section, we start by introducing notation for GPs conditioned on data and recall some of their properties, then we highlight how the smoothness of GPs can be measured using Sobolev spaces.

\subsection{Interpolation and Regression}

Let $(\Omega, \mathcal{F},\mathbb{P})$ be a probability space and $\X \subseteq \mathbb{R}^d$. A Gaussian process \citep{Stein1999,Rasmussen2006} is a stochastic process $g: \mathcal{X} \times \Omega \rightarrow \mathbb{R}$ whose properties are captured by its \emph{mean} $m:\mathcal{X} \rightarrow \mathbb{R}$, $m(x)= \mathbb{E}[g(x,\cdot)]$, and \emph{covariance function} $k\colon\X\times\X\rightarrow\bbR$, $k(x,x')=\mathbb{E}[(g(x,\cdot)-m(x))(g(x',\cdot)-m(x'))]$. The defining property of a GP with mean $m$ and covariance $k$, denoted $g\sim\mathcal{GP}(m,k)$, is that for any finite set of points $X = \{x_{i}\}_{i=1}^{n}$, the random vector $(g(x_{1},\cdot),\ldots,g(x_{n},\cdot))^{\top}\in\mathbb{R}^{n}$ follows the multivariate normal distribution $\mathcal{N}(m_{X},k_{XX})$ with mean vector given by $m_{X} = (m(x_{1}),\ldots,m(x_{n}))^{\top}\in\mathbb{R}^{n}$ and covariance matrix $k_{XX} = (k(x_{i},x_{j}))_{1\leq i,j\leq n}\in\mathbb{R}^{n\times n}$.

The covariance function is symmetric ($k(x,x') = k(x',x) \: \forall x,x'\in\X$) and positive definite ($\forall n\in\mathbb{N}$, $a_{1},\ldots,a_{n}\in\mathbb{R}$, $\{x_{i}\}_{i=1}^n \subset \X$, $\sum_{i,j = 1}^{n}a_{i}a_{j}k(x_{i},x_{j})\geq 0$) and we shall call any function satisfying these two properties a \emph{kernel}. A GP induces a probability measure over functions which we denote $\Pi_{k}$. A significant advantage of GPs over other stochastic processes is our ability to condition on data in closed form in some settings. Let $\fGP \sim \mathcal{GP}(m,k)$, $X = (x_1,\ldots,x_n)^\top$ be a finite collection of design points and for some deterministic function $f$ denote by $f_{X} = (f(x_1),\ldots,f(x_n))^\top$ the corresponding function values. Conditioning the stochastic process $\fGP$ on noisy function evaluations, often called the \textit{regression setting}, observed with independent, identically distributed Gaussian noise $\varepsilon_{i}$ with mean zero, variance $\sigma^{2}$, gives another GP, denoted $\fGP\;|\;X,y\sim \mathcal{GP}(\bar{m}_{\sigma^{2}},\bar{k}_{\sigma^{2}})$, where $y_{i} = f(x_{i}) + \varepsilon_{i}$, $\bar{m}_{\sigma^{2}}(x) = m(x) + k_{xX}(k_{XX}+\sigma^{2}I_{n\times n})^{-1}(y-m_{X})$, $\bar{k}_{\sigma^{2}}(x,x') = k(x,x') - k_{xX}(k_{XX}+\sigma^{2}I_{n\times n})^{-1}k_{Xx'}$, with $k_{xX} = (k(x,x_1),\ldots,k(x,x_n))$ and $I_{n \times n}$ is an identity matrix of size $n$. This will also be the case if $f_{X}$ is observed without noise, also called the \textit{interpolation setting}, in which case the conditioned GP is denoted $\fGP\;|\; X,f_{X}\sim \mathcal{GP}(\bar{m},\bar{k})$ where $\bar{m}(x) = k_{xX}k_{XX}^{-1}(f_{X}-m_{X})$ and $\bar{k} = \bar{k}_{0}$. 

Although the expressions for $\bar{m}$ and $\bar{m}_{\sigma^{2}}$ were obtained through conditioning of a GP, they can also arise through non-probabilistic function approximation schemes. The function spaces used are the reproducing kernel Hilbert spaces (RKHS) \citep{Berlinet2004} associated with the kernel $k$ of the GP. A Hilbert space of functions on $\X$, denoted $\mathcal{H}(\X)$, with inner product $\langle\cdot,\cdot\rangle_{\mathcal{H}(\X)}$ and norm $\|\cdot\|_{\calH(\X)}$ is called a reproducing kernel Hilbert space if there exists a kernel $k$, such that the following two conditions are satisfied (i) $\forall x\in\X$ we have $k(\cdot,x)\in \mathcal{H}(\X)$, and (ii) $\forall x\in\X$ and $\forall f\in \mathcal{H}(\X)$, we have $\langle f,k(\cdot,x)\rangle_{\mathcal{H}(\X)} = f(x)$ which is called the reproducing property. By the Moore-Aronszajn theorem, the relationship between kernels and RKHS is one-to-one, so we denote the RKHS by $\mathcal{H}_{k}(\X)$ instead of $\mathcal{H}(\X)$. 

The optimisation problem for the interpolation setting is the following constrained problem
\begin{align*}
	\argmin_{g\in \mathcal{H}_{k}(\X)}\norm{g}_{\mathcal{H}_{k}(\X)}^{2} \text{ such that } g(x_{i}) = f(x_{i}) \; \forall i \in \{1,\ldots,n\}.
\end{align*} 
The optimisation problem corresponding to regression is similar but does not require the approximating function to be exactly equal to observed data at the observation points
\begin{align*}
	\argmin_{g\in \mathcal{H}_{k}(\X)}S(g,\lambda_{n},\X)  = \argmin_{g\in \mathcal{H}_{k}(\X)}\frac{1}{n}\sum_{i = 1}^{n}(g(x_{i}) - y_i)^{2} + \lambda_{n}\norm{g}_{\mathcal{H}_{k}(\X)}^{2}.
\end{align*} 
The fit at $X$ and the complexity of the approximating function are traded off using a regularisation parameter $\lambda_{n} > 0$. When $\varepsilon_{i} = 0$ $\forall i \in \{1,\ldots,n\}$, kernel regression is sometimes referred to as approximate kernel interpolation \citep{WendlandRieger2005} due to the fact that it differs from kernel interpolation as $\lambda_{n} > 0$. For further discussion regarding the relationship between kernel methods of approximating functions and GP methods, see e.g.\ \citep{Berlinet2004,SCHEUERER2013,Kanagawa2018Review}. 

To unify notation, given a function $m$, a vector $\varepsilon\in\bbR^{n}$ and $\lambda > 0$, define the function
\begin{align}\label{eq:def_approximation}
	R_{f,\lambda,\varepsilon}^{m}(x) \coloneqq m(x) + k_{xX}(k_{XX} + \lambda I_{n \times n})^{-1}(f_{X} + \varepsilon -m_{X}),
\end{align} 
then $\bar{m} = R^{m}_{f,0,0}$ and $\bar{m}_{\sigma^{2}} = R^{m}_{f,\sigma^{2},\varepsilon}$ and the functions solving the kernel interpolation and regression problems are $R^{0}_{f,0,0}, R^{0}_{f,n\lambda_{n},\varepsilon}$ respectively and for ease of notation we will drop the variables which are zero throughout the rest of the paper. 

\subsection{The Smoothness of Reproducing kernel Hilbert Spaces}

As previously mentioned, we measure the smoothness of functions using Sobolev spaces, and this smoothness will control approximation rates. For $\tau\in\bbN, q\in[1,\infty]$ and a domain $\X\subseteq\bbR^{d}$, meaning a non-empty, open, connected set, define the integer order Sobolev space $W^{\tau}_{q}(\X)$
\begin{align*}
	W^{\tau}_{q}(\X) = \big\{f\in L^{q}(\X)\colon \forall \alpha \in  \mathbb{N}^{d}\: \lvert\alpha\rvert \leq \tau, D^{\alpha}f\in L^{q}(\X)\big\},
\end{align*}
where $\bbN^{d}$ is the set of multi-indices of size $d$, $\lvert\alpha\rvert = \sum_{i=1}^{d}\alpha_{i}$ and $D^{\alpha}$ is the weak derivative operator corresponding to $\alpha$, see e.g.\ \cite{Arcangeli2011}. Sobolev spaces can also be defined for $\tau\notin\bbN$ through a standard interpolation space argument \citep{Arcangeli2011}. In particular for $\tau >d/2$, the Sobolev space $W^{\tau}_{2}(\bbR^{d})$ may be written as
\begin{align*}
		W^{\tau}_{2}(\mathbb{R}^{d}) \coloneqq \left \{f\in L^{2}(\mathbb{R}^{d})\colon \norm{f}_{W^{\tau}_{2}(\mathbb{R}^{d})}^{2}\coloneqq \int_{\mathbb{R}^{d}}\left(1+\norm{x}_2^{2}\right)^{\tau}\lvert\hat{f}(x)\rvert^{2}dx < \infty \right\},
\end{align*}
where $\hat{f}$ is the Fourier transform of $f$ and $\norm{\cdot}_2$ denotes the Euclidean norm. Our theoretical results shall apply to functions defined over $\X\subseteq\bbR^{d}$, recall the definition of $W^{\tau}_{2}(\X)$ via restriction
\begin{align*}
	W^{\tau}_{2}(\X) \coloneqq \left\{f\colon\X\rightarrow\mathbb{R}^{d} \colon \exists f^{\circ}\in W^{\tau}_{2}(\mathbb{R}^{d})\text{ such that } f^{\circ}(x) = f(x) \; \forall x \in \X \right\},
\end{align*}
with norm
\begin{align*}
    \hspace{4mm} & \norm{f}_{W^{\tau}_{2}(\X)} = \inf \left\{\norm{f^{\circ}}_{W^{\tau}_{2}\left({\bbR}^{d}\right)}\colon f^{\circ}\in W^{\tau}_{2}({\bbR}^{d})\text{ and }f^{\circ}(x) = f(x) \; \forall x \in \X \right\}.
\end{align*}

Similarly, starting from $\calH_{k}(\bbR^{d})$, we may define $\calH_{k}(\X)$ via restriction. This function space is still an RKHS with the kernel being the restriction of $k$ to $\X\times\X$ \citep[Theorem 6]{Berlinet2004}. If $\calH_{k}(\bbR^{d})$ is norm equivalent to $W^{\tau}_{2}(\bbR^{d})$ and $\X$ is regular in some sense to be outlined in Section \ref{sec:assumptions}, then $\calH_{k}(\X)$ is norm equivalent to $W^{\tau}_{2}(\X)$. We call a kernel $\tau$\emph{-smooth} if $\calH_{k}(\X)$ is norm equivalent to $W^{\tau}_{2}(\X)$.

A frequently used example of $\tau-$smooth kernel is the Mat\'ern kernel. For $\tau > d/2$, it is given by
\begin{align}
	& k_{\Mat}(x,x')   =\nonumber\\ & \qquad\qquad\frac{2^{1-(\tau-\frac{d}{2})}A}{\Gamma(\tau-\frac{d}{2})} \left(\sqrt{2\left(\tau-\frac{d}{2}\right)}\frac{\norm{x-x'}_{2}}{l}\right)^{\tau-\frac{d}{2}}  K_{\tau-\frac{d}{2}} \left(\sqrt{2\left(\tau-\frac{d}{2}\right)}\frac{\norm{x-x'}_{2}}{l}\right),
\end{align}
where $l > 0, A > 0$. 
Here, $\Gamma$ is the Gamma function and $K_{\tau-d/2}$ is the modified Bessel function of second kind of order $\tau-d/2$. The parameter $l$ is called the lengthscale, $A$ is the amplitude. If $\tau = m + 1/2 + d/2$ for some $m\in\mathbb{N}$ then the expression drastically simplifies thanks to properties of Bessel functions \citep{Kanagawa2018Review}. Another kernel which has RKHS norm equivalent to a Sobolev space is the Wendland kernel \citep[Chapter 9]{Wendland2005}. This kernel is popular in the SDA literature due to the fact that it is compactly supported and thus offers favourable computational advantages. Both these kernels are translation invariant meaning there exists a function $\phi$ such that $k(x,y) = \phi(x-y)$.

\section{Experimental Setting}\label{sec:assumptions}

We now highlight assumptions on the experimental setting for which our theoretical results hold. Section \ref{sec:point_selection} outlines properties of the domain over which the approximation occurs and of the points at which the target function is evaluated, Section \ref{sec:assumptions_GP} outlines properties of the GP model and Section \ref{sec:literature_review} compares our assumptions to those in related literature.

\subsection{The Experimental Design} \label{sec:point_selection}

Throughout this paper, we will follow \citet{Arcangeli2011} and assume the domain $\X$ is bounded and satisfies an $(R,\delta)$ interior cone condition with a Lipschitz boundary. Such domains will be called $\mathcal{L}(R,\delta)$-domains, see Section \ref{sec:assumptions_domain} in the Appendix for full details. This is a standard assumption to make when applying scattered data approximation type results \citep{Kanagawa2019,Arcangeli2011,Teckentrup2019,Narcowich2006}. As discussed by \citet{Stein1970}, any open bounded convex set in $\bbR^{d}$ has Lipschitz boundary. This includes for example any open hypercube $(0,1)^d$ and indeed any hyper cuboid. An example of a non-Lipschitz boundary is a domain of two polygons with boundaries touching at only one point.

The experimental design problem is well studied for GP surrogate models \citep{Sacks1989,Santner2018} and an intuitive requirement is that the point set $X$ somehow covers the whole domain $\X$. Designs based on this rule-of-thumb are usually referred to as space-filling designs, see the review by \citet{Pronzato2012}.

 Given a bounded set $\X\subseteq\mathbb{R}^{d}$ and a collection of points $X\subseteq\X$, the \emph{fill distance} $h_X$, \emph{separation radius} $q_X$ and \emph{mesh ratio} $\rho_X$ are defined as
	\begin{align*}
		h_{X} \coloneqq\sup_{x\in\X}\inf_{y\in X}\norm{x-y}_2,	\qquad
		q_{X} \coloneqq \min_{\substack{x,y\in X\\x\neq y}}\frac{1}{2}\norm{x-y}_2, \qquad
		\rho_{X} = \frac{h_{X}}{q_{X}}.
	\end{align*}
	
A small fill distance guarantees that no point in the domain $\X$ is too far away from a point in the design $X$, while a large separation radius guarantees that points in the design $X$ are not too close to one another and the mesh-ratio measures the uniformity of the points. 
All of our bounds will be expressed in terms of these quantities. A sequence of points sets $\{X_{n}\}_{n\in\bbN}$ is said to be \emph{quasi-uniform}, if $\exists\: C > 0$ such that $C q_{X_{n}} \geq h_{X_{n}}$ $\forall n\in\mathbb{N}$. Note that quasi-uniformity is equivalent to a bounded mesh-ratio $\rho_{X{n}}$. Quasi-uniform points achieve optimal rates for the fill distance on $\mathcal{L}(R,\delta)$-domains, namely \citet[Satz 2.1.7]{Muller2008} showed that $\exists C_1,C_2 > 0$ such that $C_1 n^{-1/d} \leq h_{X_n} \leq C_2 n^{-1/d}$ $\forall n \in \mathbb{N}$. We now provide several examples of point sets for which results on the fill distance or separation radius are available
\begin{itemize}

	\item Regular grid points in a hypercube $\X = (0,1)^d$ form a quasi-uniform point set \citep{Johnson1990}.

	\item Random points sampled according to some probability measure on $\X$ can be shown to decrease the fill distance at a near-optimal rate in expectation. Indeed, \citet{Oates2016} showed that on a $\mathcal{L}(R,\delta)$-domain, for any $\epsilon > 0$, $\mathbb{E}[h_{X_{n}}] = \mathcal{O}(n^{-1/d + \epsilon})$ whenever the density $p>0$ on all of $\X$. 

	\item Points chosen in a restricted greedy fashion to minimise the GP posterior variance for a $\tau$-smooth kernel with $\tau > d/2+1$ result in quasi-uniform points \citep{Wenzel2019}. 

	\item Another possible choice are quasi-Monte Carlo (QMC) point sets. \sloppy Since quasi-uniformity as defined above is not studied in QMC, it is unclear when common QMC point sets are quasi-uniform. However, several special cases are known, see e.g.\ \citep{Breger2018} for quasi-uniform QMC point sets on compact Riemmanian manifolds.

	\item Some design schemes aim to minimise energy functionals. For the case of the Riesz energy, \citet{Hardin2012} showed that minimum energy point sets on compact metric spaces can be quasi-uniform.

	\item The seminal work of \citet{Johnson1990} termed points globally minimising $h_X$ ``minimax-distance designs'', and points globally maximising $q_X$ ``maximin-distance designs''. 

\end{itemize}
There are several popular choices of point sets for which exact rates for $h_{X_n}$ or $q_{X_n}$ are unknown, but which minimise these quantities numerically. The bounds in our paper clearly motivate these designs. We now present several examples:
\begin{itemize}

	\item Smolyak sparse grids, which originate from the partial differential equations literature, are also popular in the GP literature. It was shown by \citet[Theorem 3.9]{Teckentrup2019} that these points are marginally quasi-uniform when projected onto the coordinate axis, but these will not be quasi-uniform in general.
	
	\item Latin hypercube designs (LHDs) \citep{McKay1979}. Unfortunately, these are not necessarily quasi-uniform point sets. However, several authors have proposed what they call maximin and minimax LHDs \citep{Morris1995,RoshanJoseph2008,Wang2018}, which search the space of LHDs for a design optimising the fill distance or separation radius. 

	\item Many designs are model-based, the point sets depend on properties of the GPs. Two popular examples include D-optimal designs, which aim to minimise the differential Shannon information, and G-optimal designs which are selected to minimise the maximum variance of the predicted values. It was shown by \citet{Johnson1990} that these choices are asymptotically equivalent to minimax or maximin design when taking a radial kernel with lengthscale going towards zero. 

\end{itemize}

\subsection{The Gaussian Process Model and Hyperparameter Selection}\label{sec:assumptions_GP}

Let $m(\theta)$ and $k(\theta)$ denote the mean function and covariance kernel in the GP model parameterised by some $\theta\in\Theta\subset\bbR^{d_{\Theta}}$. In practice, it is common to learn hyperparameters as more data points are observed, and our convergence results will allow for such adaptivity. There exists a vast literature on parameter estimation for GPs; for an overview, see e.g.\ \citep[Chapter 6]{Stein1999}, which includes a detailed discussion of Mat\'ern kernels, or \citep[Chapter 5]{Rasmussen2006}. 

For the mean function $m(\theta)$, it is common to use a parametric model whose parameters are estimated using least-squares. Of course, other methods, such as empirical-risk minimisation and gradient-based optimisation could also be used. For the covariance function $k(\theta)$, parameters controlling lengthscales, amplitudes and smoothness need to be estimated. Common approaches include maximum marginal likelihood estimation, sometimes referred to as empirical Bayes, and cross-validation. In Bayesian settings, it is also common to provide a full prior on these hyperparameters and consider a predictive distribution taking into account uncertainty in the parameters. 

Our bounds will be independent of the method used for parameter estimation, following the approach of \citet{Teckentrup2019}. The convergence rates will only depend on how the smallest and largest smoothness of the approximation function $R^{m}_{f,\lambda,\varepsilon}$ for $\theta\in\Theta$ and the corresponding norm-equivalence constants. For this reason, we will use the notation $R^{m}_{f,\lambda,\varepsilon}(\theta)$ to emphasise the dependence on the parameter values. If $k(\theta)$ is $\tau(\theta)$-smooth, then we denote the norm equivalence constants by
\begin{align}\label{eq:norm_equivalence_constants}
	C_{l}(\theta)\norm{\cdot}_{\calH_{k(\theta)}(\X)}\leq \norm{\cdot}_{W^{\tau(\theta)}_{2}(\X)}\leq\ C_{u}(\theta)\norm{\cdot}_{\calH_{k(\theta)}(\X)}.
\end{align}
Assume that the parameter estimation method gives a sequence of hyperparameters $\{\theta_{n}\}_{n=1}^{\infty}$ so that once the $n$-th data point has been observed, the parameters $\theta_{n}$ are used. Following \citet{Teckentrup2019} given $N\in\bbN$ define $\tau^{-}_{k}\coloneqq\inf_{n\geq N}\tau(\theta_{n})$, $\tau^{+}_{k}\coloneqq\sup_{n\geq N}\tau(\theta_{n})$ and $C_{N} = \sup_{n\geq N}C_{u}(\theta_{n})C_{l}(\theta_{n})^{-1}$. This set of extreme values is denoted by $\Theta_{N}^{*} = \{\tau_{k}^{+},\tau_{k}^{-},C_{N}\}$. These quantities represent the extremes of the smoothness of the kernel and the ratio of norm equivalence constants after the $N$-th data point has been observed. Of course, these parameters are often selected as data is observed. As a result, to bound expressions for $n\geq N$, we need to ensure the parameter selection methods used does not result in extreme values regardless of the data observed. We note that in the context of Gaussian regression, the observation noise parameter $\sigma$ could also be estimated from data, leading to a sequence of parameters $\{\sigma_{n}\}_{n=1}^\infty$. A common approach is to maximise the marginal likelihood.

\subsection{Comparison to Related Literature}\label{sec:literature_review}

Now that we have discussed our experimental setting, we briefly remark on connections with related literature using kernel approximations. In our work, the target function is modelled as an unknown deterministic function, possibly corrupted by noise, with no assumption on the distribution of the design points. The error bounds shall be expressed in terms of the smoothness of the approximating function, the smoothness of the true function, and geometric properties of the design points. 

The closest approach to the work in this paper can be found in the scattered data approximation \citep{Wendland2005} literature. Indeed, our proofs harness multiple results from the field. The main difference is that we tackle the combination of corrupted data, misspecified smoothness and misspecified likelihood, whereas existing works have only covered these cases individually. Examples include approximate interpolation \citep{WendlandRieger2005}, deterministic corruption \citep{Rieger.Zwicknagl:2009} and random error satisfying a regularity condition \citep{Arcangeli2011,Utreras1988}. A framework for managing smoothness misspecification was presented by \citet{Narcowich2006} which uses quasi-uniform point placement. Adapting hyperparameters with no observation corruption was studied by \citet{Teckentrup2019}. 

Statistical learning theory \citep{Steinwart2008,Cucker2007} takes the view of approximation as an optimisation problem in an RKHS, outlined in Section \ref{sec:background}, with the target function specified by some joint probability distribution on the input and output spaces. A sampling distribution for the location of the data points is assumed, which is employed as the weight measure for the norm used to measure error of the approximation. This statistical assumption is the main difference with the SDA view, which we use, since in SDA the error bounds are expressed in terms of the experimental design. Additionally the remedy for smoothness misspecification in SLT is altering the parameters of the approximating function \citep{Steinwart2009} as opposed to quasi-uniform points. 

Nonparametric regression \citep{Gyrfi2002,Wahba1990} is an approach to regression which assumes no parametric underlying form for the target function. Such techniques bare a lot of resemblance to SDA and statistical learning theory, and indeed have similar methods for obtaining approximating functions. Sometimes the unknown function is assumed to be a draw from a distribution over a space of functions, for example in Kriging \citep{Matheron1963,Stein1999}. This is clearly different from the SDA paradigm which assumes the quantity of interest is a fixed deterministic function. As done in statistical learning theory a sampling distribution of data locations is also often assumed. Within this nonparametric paradigm, an important subclass is Bayesian nonparametric regression \citep{Ghosal2017,Gine2016}. These take the Bayesian view of modelling by placing a prior measure on the unknown target function, and using a likelihood and Bayes' rule to obtain a posterior measure on the unknown quantity given observed data. Contraction of the entire posterior measure is studied which is stronger than contraction of the posterior mean function, the focus of Section \ref{sec:theory}. Again the assumption of a sampling distribution of the points and the method of dealing with smoothness misspecification makes this modelling paradigm distinct from the one considered in this paper.

\section{Convergence Guarantees for Gaussian Process Means} \label{sec:theory}

We are now ready to present the main results of the paper. All of the proofs are provided in the appendix. We will use the following notation $x \wedge y = \min(x,y)$, $x \vee y = \max(x,y)$, $(x)_+ = \max(x,0)$. $\lfloor x \rfloor$ denotes the integer part of $x$ and $\lceil x \rceil$ the ceiling of $x$. The integrability parameter in the Sobolev norms will be $q\in[1,\infty]$. Following \citet{Arcangeli2011} define $\tau_{0} \coloneqq \tau - d(1/2-1/q)_{+}$ and $\tau^{*} \coloneqq \tau_{0}$ if $\tau\in\mathbb{N}$ and either $2<q<\infty$ and $\tau_{0}\in\mathbb{N}$, or $q=2$, else we will have $\tau^{*} \coloneqq \lceil\tau_{0}\rceil -1$. Finally, for $a,b > 0$, let $\tilde{a} = a  -\lfloor a \rfloor$ and define: $\Lambda_{a,b}\coloneqq (b\tilde{a}(1-\tilde{a}))^{1/b}$, if $\tilde{a}\in(0,1)$ and $\Lambda_{a,b}\coloneqq1$ if $\tilde{a} = 0$.

\subsection{Convergence Guarantees for Interpolation} \label{sec:noiseless_convergence}

This section considers approximations with noiseless function evaluations observed at a finite collection of $n$ points $X_{n}\subset\X$. We will assume that the likelihood is well specified in that the data is indeed noiseless. The interpolation setting is of particular interest since it leads to a closed-form approximation, and corresponds to the use of GPs for range of applications including to computer models \citep{Kennedy2001}, Bayesian inverse problems \citep{Teckentrup2019} and Bayesian numerical methods \citep{Bull2011,Xi2018MultiOutput,Briol2019,Chen2019}. From a practical point of view, the result provide insights into point-picking strategies and hyperparameter selection for these applications. Before stating the first bound, we summarise all of the necessary assumptions which were mentioned in the previous section

\begin{assumption}[\textbf{Assumptions on the Domain}]\label{assumption:domain}
	$\X$ is an $\mathcal{L}(R,\delta)$-domain for \sloppy some $R>0$ and $\delta \in (0,\pi/2)$.
\end{assumption}
\begin{assumption}[\textbf{Assumptions on the Kernel Parameters}]\label{assumption:parameters}
	\sloppy Given $N\in\bbN$, for $n\geq N$,  $k(\theta_{n})$ is $\tau(\theta_{n})$-smooth and the elements of $\Theta^{*}_{N}$ are finite with $\tau^{-}_{k}> d/2$. 
\end{assumption}

\begin{assumption}[\textbf{Assumptions on the Kernel Smoothness Range}]\label{assumption:smoothness}
	\sloppy Given $N\in\bbN$, the set $\{\tau(\theta_{n})\}_{n\geq N}$ has finitely many values. 
\end{assumption}

\begin{assumption}[\textbf{Assumptions on the Target Function and Mean Function}]\label{assumption:mean}

\sloppy The target function satisfies $f\in W^{\tau_{f}}_{2}(\X)$ for some $\tau_{f} > d/2$ and given $N\in\bbN$ the mean function satisfies $\sup_{n\geq N}\norm{m(\theta_{n})}_{W^{\tau_{f}}_{2}(\X)}<\infty$.
\end{assumption}

Assumption \ref{assumption:domain} ensures that the domain is sufficiently regular to use extension and embedding theorems. For a discussion about examples of domains satisfying the assumptions see Section \ref{sec:assumptions_domain}. Assumption \ref{assumption:parameters} ensures that the RKHS of $k(\theta_{n})$ is norm equivalent to a Sobolev space with smoothness $\tau(\theta_{n})$ and that the parameters for the model are not so extreme as to result in arbitrarily smooth or arbitrarily rough functions. This assumption also concerns the ratio of the norm equivalence constants to ensure that their ratio is finite. For the case of $k$ being a Mat\'ern kernel, a sufficient condition was given by \citet[Lemma 3.4]{Teckentrup2019} which shows that $C_{N}\leq \sup_{n\geq N}\max(l_{n},l_{n}^{-1})$ where $l_{n}$ is the lengthscale when using parameter setting $\theta_{n}$.

The $N$ term facilitates a ``burn-in'' period for narrowing down the desired range of hyperparameters. Assumption \ref{assumption:smoothness} is required in order to provide a uniform bound over parameter values. The assumption is satisfied in the common scenario where cross validation is used for smoothness parameter selection where there is a finite candidate set of smoothness parameters. For example, the widely used Mat\'ern kernel has a convenient closed form for $\tau = m + 1/2 + d/2$ for $m\in\bbN$, whereas other smoothness level require evaluations of Bessel functions which is computationally challenging. In practice, it is therefore very common to focus on $\{\tau(\theta_{n})\}_{n\geq N} = \{m+1/2+d/2\}_{m\in M}$ for some finite set $M\subset\bbN$. We note that Assumption \ref{assumption:smoothness} is not required by \citet{Teckentrup2019} since weaker sampling inequalities depending only on the integer part of the smoothness parameter were used in that paper. Assumption \ref{assumption:mean} ensures that the target function has a minimal level of regularity and that the parameterised mean function used in the prior GP is at least as smooth as the target function.
 
We are now ready to state our main result for GP interpolation. This will be split into two parts covering the well-specified ($\tau_{f}\geq \tau^{+}_{k}$) and misspecified ($\tau_{f} < \tau^{+}_{k}$) smoothness settings. 

\begin{theorem}\label{thm:Misspecified_Sobolev_Interpolation}
Fix $N\in\bbN$ and suppose Assumptions 1-4 hold. Let $q\in[1,\infty]$ and $s\in[0,(\tau_{f}\wedge\tau^{-}_{k})^{*}]$. Then, $\exists  C_{0},h_{0} > 0$ such that $\forall n\geq N$, $\forall X_{n}\subseteq\X$ with $h_{X_{n}}\leq h_{0}$, when $\tau_{f}\geq \tau^{+}_{k}$
\begin{align*}
		&\bignorm{f-R^{m}_{f}(\theta_{n})}_{W^{s}_{q}(\X)}  \leq Ch_{X_n}^{\tau^{-}_{k} - s - d \left(\frac{1}{2} - \frac{1}{q}\right)_{+}}\left(\norm{f}_{W^{\tau_{f}}_{2}(\X)} + \sup_{n\geq N}\bignorm{m(\theta_{n})}_{W^{\tau_{f}}_{2}(\X)}\right),
\end{align*}
 and when $\tau_{f} < \tau^{+}_{k}$
\begin{align*}
		&\bignorm{f-R^{m}_{f}(\theta_{n})}_{W^{s}_{q}(\X)}  \leq Ch_{X_n}^{(\tau_{f}\wedge\tau^{-}_{k}) - s - d \left(\frac{1}{2} - \frac{1}{q}\right)_{+}}\rho_{X_{n}}^{(\tau^{+}_{k}-\tau_{f})}\left(\norm{f}_{W^{\tau_{f}}_{2}(\X)} + \sup_{n\geq N}\bignorm{m(\theta_{n})}_{W^{\tau_{f}}_{2}(\X)}\right),
\end{align*}
where $C = C_{0}\Lambda_{s,q}$ with $C_{0} = C_{0}(\X, d,\tau_{f},q,\Theta^{*}_{N})$ and $h_{0} = h_{0}(R,\delta,d,\tau_{f},\Theta^{*}_{N})$. 
\end{theorem}

This theorem is an extension of the result by \citet[Theorem 3.5]{Teckentrup2019} since it holds for a wider range of target functions $f$. In particular, it only requires $\tau_{f}>d/2$ rather than $\lfloor\tau_{f}\rfloor>d/2$, as such, alleviates the issues mentioned by \citet[Remark 3.7]{Teckentrup2019}. The range of the smoothness parameter $s$ in the norm is dictated by $\tau_{f}$, the smoothness of the target function, and $\tau_{k}^{-}$, the minimum smoothness of the approximating function. There is a large freedom in the norm choice, for example a bound for $L^{2}$ approximation can be recovered by setting $s = 0, q = 2$, and $L^\infty$ is recovered with $s=0, q= \infty$. We will see in Section \ref{sec:bayesian_pn} that this flexibility can be useful for applications.

The upper bound holds only when the data points provide a sufficient initial covering of $\mathcal{X}$, as measured via the $h_{0}$ term, see e.g\ \citep[Remark 3.2]{Arcangeli2011} for a discussion of $h_0$. The behaviour of the constant $\Lambda_{s,q}$ is discussed further by \citet[Section 4.2]{Arcangeli2011}. Aside from the exponent of $h_{X_{n}}$, $\Lambda_{s,q}$ is the only term on the right hand side that depends on $s$, therefore the same $C_{0}$ value can be used for different $s$ values. We now highlight how the bound depends on model-specific choices.

\begin{itemize}

\item \emph{Experimental design:} The terms $h_{X_{n}}$ and $\rho_{X_{n}}$ quantify the impact of the experimental design. A detailed discussion of these quantities was provided in Section \ref{sec:point_selection}. In general, the approximation error bound is always minimised by making $h_{X_{n}}$ and $\rho_{X_{n}}$ as small as possible. We recall that the optimal decay of $h_{X_{n}}$ is $n^{-\frac{1}{d}}$ and the optimal case for $\rho_{X_{n}}$ is when it is bounded by a constant independent of $n$. Both of these properties occur when quasi-uniform points are used, and this is therefore a reasonable criterion for point selection. When quasi-uniform points are used, the optimal error rate is obtained in terms of worst case complexity \citep[Theorem 4.17]{Novak2008}.

\item \emph{Kernel smoothness:} The rate of convergence, as a function of $h_{X_n},q_{X_{n}},\rho_{X_{n}}$, is controlled by $\tau_f, \tau_k^+$ and $\tau_k^-$. In general, the larger the value of $\tau_{f}$, the faster the convergence rate can be.  Two regimes are highlighted. When $\tau_{f} < \tau_{k}^{+}$, meaning smoothness is misspecified, then $(\tau_{k}^{+}-\tau_{f})$ penalises overestimation of $\tau_{f}$ by increasing the exponent of $\rho_{X_{n}}$. Therefore, if one believes they are in danger of over estimating the smoothness of the true function, then quasi-uniform points should be used. When $\tau_{f}\geq\tau_{k}^{+}$, we see $\tau_{k}^{-}$ penalises underestimation of $\tau_{f}$ by limiting the exponent of $h_{X_{n}}$.

\item \emph{Other kernel parameters:} The bound can also be helpful when it comes to understanding the impact of adapting hyperparameters which do not change the smoothness of the RKHS. For those, adaptively choosing the hyperparameters does not impact the rate of convergence in $n$, but only constants of the bound. Indeed, It can be seen in the proof that $C_{0}$ depends on the extremes of the norm equivalence constants. 
\end{itemize}

\subsection{Convergence Guarantees for Regression with Gaussian Likelihood} \label{sec:noisy_convergence}

This section considers observations that are corrupted with independently and identically distributed Gaussian noise so the data is $y_i = f(x_i) + \varepsilon_i$ where $\varepsilon_i \sim \mathcal{N}(0,\sigma^2)$. Once again, the mean of the GP conditioned on this data is available in closed-form, and a well-specified likelihood is used. Three further assumptions are required. 

\begin{assumption}[\textbf{Additional Assumptions on Kernel Parameters}]\label{assumption:double_SI_application}
	\sloppy Given $N\in\bbN$ and $\tau_{f} > d/2$ for all $n\geq N$ we have $\tau(\theta_{n})\in (d/2,\tau_{f}]\cup[\lceil\tau_{f}\rceil,\infty)$.
\end{assumption} 
\begin{assumption}[\textbf{Assumption on Small Ball Probabilities}]\label{ass:small_ball}
	Given $N\in\bbN$, $\exists  c,\alpha_{N} > 0$ such that $\Pi_{k(\theta_{n})}\left(\norm{f}_{L^{\infty}\left(\overline{\X}\right)}\leq c \right) \leq \exp(-\alpha_{N})$ $\forall n\geq N$.
\end{assumption}
\begin{assumption}[\textbf{Additional Assumption on the Target Function}]\label{ass:target_extension}
Given $\tau > d/2$, $f$ has an extension $f^{\circ}\in C^{\tau_{f}}(\mathbb{R}^{d})\cap  W^{\tau_{f}}_{2}(\mathbb{R}^{d})$ where $C^{\tau_{f}}(\mathbb{R}^{d})$ is the space of $\tau_{f}$ H\"older continuous functions. 
\end{assumption}
Assumption \ref{assumption:double_SI_application} restricts slightly the smoothness values that $f$ can take. It is required due to the double use of a sampling inequality in our proof, see the proof of Theorem \ref{thm:Misspecified_Sobolev_Det_Regression} for further explanation. This is not a very restrictive assumptions since the length of interval containing disallowed values is less than one. Assumption \ref{ass:small_ball} involves the measure on functions induced by the GP with parameters $\theta$ and ensures the size of the GP samples cannot be uniformly small with arbitrarily high probability, since this would result in a somehow degenerate GP. This assumption is implicitly used by \citet[Theorem 1.2]{Linde1999} which is a key auxiliary result for \citet[Theorem 1]{VanderVaart2011}, which our proof follows closely. In the case where an amplitude parameter is used for the kernel (e.g. $A$ for the Mat\'ern kernel in Section \ref{sec:background}), the assumption is satisfied if this parameter is bounded away from zero. This can be seen by using concentration inequalities for the supremum of a Gaussian process, see e.g.\ \cite[Chapter 2.4]{Gine2016}. It should be noted that the commonly used maximum likelihood procedure can result in $A$ decaying to zero \citep{Karvonen2020}. Assumption \ref{ass:target_extension} concerns the regularity of the target function. This is a requirement for Lemma 4 of \citet{VanderVaart2011} which is an auxiliary result that is employed in our bound, see Appendix \ref{appendix:gaussian_noise} for details.

 We can now present our main theorem for GP regression, which is stated in expectation over the distribution of the noise. Again, we separate the well-specified and misspecified smoothness settings.
\begin{theorem}\label{thm:Misspecified_Sobolev_Random_Regression}
	Fix $N\in\bbN$ and suppose Assumptions 1-\ref{ass:target_extension} hold. Let $q\in[1,\infty]$ and $s \in [0,(\tau_{f}\wedge\tau^{-}_{k})^{*}]$. Then, $\exists  C,h_{0} > 0$ such that $\forall n\geq N$, $\forall X_{n}\subseteq\X$ with $h_{X_{n}}\leq h_{0}$, when $\tau_{f}\geq \tau_{k}^{+}$
	\begin{align*}
		& \mathbb{E}\left[ \bignorm{f-R^{m}_{f,\sigma^{2},\varepsilon}(\theta_{n})}_{W^{s}_{q}(\X)}\right]\\
		& \leq Ch_{X_{n}}^{\frac{d}{\gamma}-s}\bigg[h_{X_{n}}^{\tau^{-}_{k}-\frac{d}{2}}\left(\norm{f}_{W^{\tau_{f}}_{2}(\X)} + \sup_{n\geq N}\bignorm{m(\theta_{n})}_{W^{\tau_{f}}_{2}(\X)}\right) + n^{\frac{1}{2}} h_{X_{n}}^{\tau^{-}_{k}-\frac{d}{2}} + 
		n^{\frac{d}{4\tau^{-}_{k}}}\bigg],
	\end{align*}
	and when $\tau_{f} < \tau_{k}^{+}$
	\begin{align*}
		& \mathbb{E}\left[ \bignorm{f-R^{m}_{f,\sigma^{2},\varepsilon}(\theta_{n})}_{W^{s}_{q}(\X)}\right]\\
		& \leq Ch_{X_{n}}^{\frac{d}{\gamma}-s}\bigg[h_{X_{n}}^{\left(\tau_{f}\wedge\tau^{-}_{k}\right)-\frac{d}{2}}\rho_{X_{n}}^{\left(\tau^{+}_{k}-\tau_{f}\right)_{+}}\left(\norm{f}_{W^{\tau_{f}}_{2}(\X)} + \sup_{n\geq N}\bignorm{m(\theta_{n})}_{W^{\tau_{f}}_{2}(\X)}\right) \\
		& \qquad\qquad\qquad+ n^{\frac{1}{2}} h_{X_{n}}^{\tau^{-}_{k}-\frac{d}{2}} + 
		n^{\left(\frac{1}{2}-\frac{\tau_{f}}{2\tau^{+}_{k}}\right)_{+}\bigvee \left(\frac{d}{4\tau^{-}_{k}}\right)}\bigg],
	\end{align*}
	where $C=C_{0}\Lambda_{s,q}$ with $C_{0} = C_{0}\left(\X,d,\tau_{f},q,\norm{f}_{W^{\tau_{f}}_{2}(\X)},\sup_{n\geq N}\norm{m(\theta_{n})}_{W^{\tau_{f}}_{2}(\X)},\Theta_{N}^{*}\right)$, $h_{0} = h_{0}\left(R,\delta,d,\tau_{f},\Theta^{*}_{N}\right)$ and $\gamma = 2 \vee q$. 
\end{theorem}
As far as we are aware this is the first combination of SDA and Bayesian nonparametrics techniques. The closest result that we know of is by \citet[Theorem 8.1]{Arcangeli2007} but does not cover the present scenario due to that result having requirements on the noise not satisfied by Gaussian noise. 

The bounds contain a sum of three terms. The first term gives a rate identical to the interpolation case, and the later two describe the impact of the Gaussian noise. These last two terms will usually decrease to zero at a slower rate in $n$, and again, one can notice a clear advantage of using quasi-uniform points. The dependence on the norms of $f$ and $m$ in $C_{0}$ arises from the use of the result by \citet[Lemma 4]{VanderVaart2011}. This dependence is made explicit in the proof and occurs in a small-ball probability bound. Due to Assumption \ref{assumption:double_SI_application}, there is a limitation in our theory for $d=1$. Specifically, $\tau_{f} + d/2$ could be smaller than $\lceil\tau_{f}\rceil$ when $d = 1$ so Assumption 4 might not be satisfied. But in two dimensions and higher, Assumption \ref{assumption:double_SI_application} does not impose extra restrictions since then $\tau_{f} + d/2 \geq \lceil\tau_{f}\rceil$ so $\tau(\theta_{n}) = \tau_{f} + d/2$ is a permissible value. 

We once again comment on the impact of model choice. The advice in terms of experimental design is once again to use quasi-uniform points. The main difference with the previous theorem is for the smoothness parameters of the kernel.

\begin{itemize}

\item \emph{Kernel smoothness:} The equality $\tau_{f} + d/2 = \tau^{+}_{k} = \tau^{-}_{k}$ optimises the bound when using quasi-uniform points. This corresponds to the sample paths of the GP matching the smoothness of the target function \citep{Kanagawa2018Review,Lukic2001}. This is a phenomenon that occurs in this setting due to the true and assumed likelihood both being Gaussian, which is in distinct contrast to the interpolation case where the bound is optimised when $\tau_f = \tau_k^+ = \tau_k^-$. This choice of smoothness parameter might seem unintuitive from the perspective of kernel ridge regression. However this can be rationalised by observing that the connection to kernel ridge regression can only be made if the regularisation being kept constant and not altering with added data. An in depth discussion is provided by \citet[Section 5.1]{Kanagawa2018Review}.

\end{itemize}

\begin{corollary}\label{cor:gaussian_noise_cor}
	Fix $N\in\bbN$ and suppose Assumptions 1-\ref{ass:target_extension} hold with $\tau_{k}^{+} = \tau_{k}^{-} = \tau_{f} + d/2$. Let $q\in[1,2]$ and $s \in [0,\tau_{f}^{*}]$. Then, $\exists  C,h_{0} > 0$ such that $\forall n\geq N$, $\forall X_{n}\subseteq\X$ quasi-uniform with $h_{X_{n}}\leq h_{0}$
 	\begin{align*}
 		\mathbb{E}\left[ \bignorm{f-R^{m}_{f,\sigma^{2},\varepsilon}(\theta_{n})}_{W^{s}_{q}(\X)}\right] \leq Cn^{-\frac{\tau_{f}}{2\tau_{f} + d} + \frac{s}{d}},
 	\end{align*}
 	where $C= C_{0}\Lambda_{s,q}$ with $C_{0} = C_{0}\left(\X,d,\tau_{f},q,\norm{f}_{W^{\tau_{f}}_{2}(\X)},\sup_{n\geq N}\norm{m(\theta_{n})}_{W^{\tau_{f}}_{2}(\X)},\Theta_{N}^{*}\right)$ and $h_{0} = h_{0}\left(R,\delta,d,\tau_{f},\Theta^{*}_{N}\right)$.
\end{corollary}

When $q=2,s=0$ this is the mini-max optimal rate for non-parametric regression within the Bayesian nonparametric paradigm, see e.g.\ \citep[Chapter 2]{Tsybakov2009} and references therein. A comparison can be made to a recent result in statistical learning theory \cite[Corollary 6]{Fischer2017} which has $s/(2\tau+d)$ instead of $s/d$ meaning it is stronger than our result. However, as discussed in Section \ref{sec:assumptions}, the assumptions in the statistical learning paradigm is somewhat different to our setting as we do not consider a norm weighted by the point sampling distribution.

\subsection{Convergence Guarantees with Misspecified Likelihoods} \label{sec:fixednoise_convergence}

Now that we have presented our results for well-specified likelihoods, we extend these to the misspecified case. We recall that GP approximations based on interpolation or Gaussian likelihoods are often used due to their closed form expressions, but that these are often idealisations of the problem. 

This section illustrates the impact of this idealisation on convergence. In each case, the bound allows for arbitrary corruption $y_i = f(x_i) + \varepsilon_i$ where $\{\varepsilon_i\}_{i=1}^n$ do not have to be i.i.d.\ nor Gaussian, and could even be deterministic. The corruption is manifested in the bounds only in a $\bbE[\norm{\varepsilon}_{2}]$ term. The main point of this section is that quasi-uniform points are not only essential for smoothness misspecification, but can also be of significant help to counter likelihood misspecification. Due to the lack of assumptions on the type of corruption, our bounds should be interpreted as worst-case type bounds. 

They are particularly suitable for misspecification models studied in the robust regression literature \citep{Rousseeuw1987,Huber2009,Christmann2007} in particular \citet{Christmann2007} studies the case of kernel ridge regression. For example, bias robustness corresponds to the setting where the i.i.d Gaussian assumption is satisfied up to a small number of corruptions, usually called outliers. This is common when the data are collected from physical or medical sensors as these tend to have faults after a certain period of time, see e.g.\ \citep{Armstrong1988}. It also occurs in many applications of Bayesian optimisation \citep{Martinez-Cantin2018}.  If a fixed number of data points is contaminated by outliers, then $\|\varepsilon\|_2 = O(1)$. Alternatively, it could be that some proportion of the total number of data points is corrupted. For example, if $n^{\alpha},\alpha \in (0,1)$ data points are corrupted, then $\|\varepsilon\|_2 = O(n^{\alpha/2})$, whereas if $\beta n, \beta \in (0,1]$ data points are corrupted, then $\|\varepsilon\|_2 = O(\sqrt{ \beta n})$.  

Another possible scenario is the (pessimistic) case of arbitrary random and unbounded noise, see e.g. \citep{Stegle2008}. In this case, assuming that $\mathbb{E}[\varepsilon_i] < \infty $ and $\mathbb{E}[\varepsilon_i^2] < \infty$ $\forall i$, we get that: $\mathbb{E}[\|\varepsilon\|_2] = O(\sqrt{n})$ regardless of the distribution of these corruptions or of any correlation. We note that the exact distribution of these norms have been derived for a range of distributions \citep{Mathai1992}. Of course, it should be possible to improve on these worst-case bounds by using stronger assumptions on the distributions of the noise terms. For example, the Laplace to Gaussian misspecification was previously studied by \cite{Kleijn2006}. It would be interesting, but beyond the scope of this paper, to combine such results with scattered data approximation bounds to produce bounds in the same fashion as Theorem \ref{thm:Misspecified_Sobolev_Random_Regression}. 

Finally, one setting where our bounds will not be of help is when the noise distribution has infinite first or second moment. In this case, $\mathbb{E}[\|\varepsilon\|_2] =\infty$ and the bounds will be vacuous. This will be the case for Cauchy noise, or for certain instances of student-t or Pareto noise.

\subsubsection{Misspecified Gaussian Regression Likelihood}

For the first result, the Gaussian likelihood $\mathcal{N}(0,\sigma_{n}^{2})$ is implicitly assumed. This corresponds to considering $R_{f,\sigma^2_n,\varepsilon}^{m}$  with $\sigma_{n} > 0$ as the approximating function. The subscript in $\sigma_{n}$ is kept since we might want to vary the parameter with $n$ in order to improve the convergence rate. This is to be interpreted as a worst-case type result since no assumption is placed upon the corruption. 

\begin{theorem}\label{thm:Misspecified_Sobolev_Det_Regression}
Fix $N\in\bbN$ and suppose Assumptions 1-\ref{assumption:double_SI_application} hold. Let $q\in[1,\infty]$, $s\in[0,(\tau_{f}\wedge\tau^{-}_{k})^{*}]$ and $\sigma_{n} > 0\: \forall n\in\bbN$. Then, $\exists  C,h_{0} > 0$ such that $\forall n\geq N$,  $\forall X_{n}\subseteq\X$ with $h_{X_{n}}\leq h_{0}$, when $\tau_{f}\geq\tau_{k}^{+}$
\begin{align*}
	\mathbb{E}\left[ \bignorm{f -R_{f,\sigma^2_n,\varepsilon}^{m}(\theta_{n})}_{W^{s}_{q}(\X)} \right] \leq Ch_{X_{n}}^{\frac{d}{\gamma}-s}
		\bigg[\left(h_{X_{n}}^{\tau^{-}_{k}-\frac{d}{2}}+ \sigma_n \right)& \left(\norm{f}_{W^{\tau_{f}}_{2}(\X)} + \sup_{n\geq N}\bignorm{m(\theta_{n})}_{W^{\tau_{f}}_{2}(\X)}\right) \\
		& + \left(h_{X_{n}}^{\tau^{-}_{k}-\frac{d}{2}}\sigma_{n}^{-1}+ 1\right)\mathbb{E}[\norm{\varepsilon}_{2}] \bigg],
\end{align*}
and when $\tau_{f} < \tau_{k}^{+}$
	\begin{align*}
		&\mathbb{E}\left[ \bignorm{f -R_{f,\sigma^2_n,\varepsilon}^{m}(\theta_{n})}_{W^{s}_{q}(\X)} \right] \\
		&  \leq Ch_{X_{n}}^{\frac{d}{\gamma}-s}
		\bigg[\left(h_{X_{n}}^{\left(\tau_{f}\wedge\tau^{-}_{k}\right)-\frac{d}{2}}\rho_{X_{n}}^{(\tau^{+}_{k} - \tau_{f})}+ \sigma_n q_{X_n}^{-(\tau_k^{+} - \tau_{f})}\right) \left(\norm{f}_{W^{\tau_{f}}_{2}(\X)} + \sup_{n\geq N}\bignorm{m(\theta_{n})}_{W^{\tau_{f}}_{2}(\X)}\right)\\
		& \qquad  \qquad \qquad  + \left(h_{X_{n}}^{\tau^{-}_{k} -\frac{d}{2}}\sigma^{-1}_n + 1\right)\mathbb{E}[\norm{\varepsilon}_{2}] \bigg],
	\end{align*}
	where $C = C_{0}\Lambda_{s,q}$ with $C_{0} =  C_{0}\left(\X,d,q,\tau_{f},\Theta_{N}^{*}\right)$, $h_{0} = h_{0}\left(R,\delta,d,\tau_{f},\Theta_{N}^{*}\right)$ and $\gamma = 2\vee q$.
\end{theorem}

This generalizes the results by \citet[Proposition 3.6]{WendlandRieger2005} and \citet[Theorem 7.1]{Arcangeli2007} to misspecified likelihood and smoothness. Assumptions \ref{ass:small_ball} and \ref{ass:target_extension}, used in Theorem \ref{thm:Misspecified_Sobolev_Random_Regression}, are not required since the corruption is not assumed Gaussian. The effect of the corruption is manifested solely in the $\mathbb{E}[\norm{\varepsilon}_{2}]$ term and to conclude the right hand side converges to zero, the growth of $\mathbb{E}[\norm{\varepsilon}_{2}]$ needs to be sufficiently bounded. The theorem leads us to a useful recommendation for $\sigma_n$ in settings where the data is noiseless. 

\begin{itemize}

\item \emph{Adaptive Likelihood/Nugget:} As noted in Section \ref{sec:assumptions}, it is common to add a ``nugget'' term to kernel matrices in order to improve numerical stability. This corresponds to taking $\sigma_n > 0$, and larger values of $\sigma_n$ are known to provide greater stability at the cost of a slower convergence rate. When there is no corruption this is referred to as approximate interpolation \citep{WendlandRieger2005}. Theorem \ref{thm:Misspecified_Sobolev_Det_Regression} provides a way of choosing $\sigma_{n}$  for this scenario. Setting $q = \infty$ and $\tau_k^-=\tau_k^+=\tau_f$, meaning we are in the well specified smoothness case, the choice $\sigma_n \propto h_{X_n}^{(\tau_f \wedge \tau_k^-)+(\tau_k^+-\tau_f) - d/2} = h_{X_n}^{\tau_f - d/2}$ optimises the bound. This coincides with the choice proven to be optimal for matrix conditioning by \citet[Corollary 3.7]{WendlandRieger2005}.

\end{itemize}
Thinking of this suggestion from the point of view of adaptive likelihood may seem unnatural at first since the likelihood is normally a fixed object which is independent of data. However, this suggestion can also be viewed from a regularisation perspective as altering the penalisation in the optimisation problem $S$, see Section \ref{sec:background}.

We now give two corollaries of Theorem \ref{thm:Misspecified_Sobolev_Det_Regression} under different assumptions on $\tau_k^-$ and $\tau_k^+$. In each case, these provide insights into model choices which optimise the bounds. First, consider $\tau_k^- =\tau_k^+ = \tau_f$, in which case the smoothness is correctly specified. The next result gives a bound when $h_{X_n}$ has the optimal rate $n^{-1/d}$ and $\sigma_{n}$ is kept constant.

\begin{corollary}\label{cor:deterministic_corrupt}
Fix $N\in\bbN$ and suppose Assumptions 1-\ref{assumption:double_SI_application} hold. Let $q\in[1,\infty]$, $s\in[0,\tau_{f}^{*}]$, $\tau_{k}^{-} = \tau_{k}^{+}=\tau_{f} $ and $\sigma_{n} = \sigma$. Assume $h_{X_{n}}\leq C_{1}n^{-\frac{1}{d}}$ for some $C_{1} > 0$. Then, $\exists C, h_{0} > 0$ such that $\forall n\geq N$ with $h_{X_{n}}\leq h_{0}$
\begin{align*}
&\mathbb{E}\left[\bignorm{f -R_{f,\sigma^2,\varepsilon}^{m}(\theta_{n})}_{W^{s}_{q}(\X)}\right] \\
		&\qquad\leq Cn^{-\frac{1}{\gamma} +\frac{s}{d}}\left(\bbE[\norm{\varepsilon}_{2}] + n^{-\frac{\tau_{f}}{d}+\frac{1}{2}}\left(\norm{f}_{W^{\tau_{f}}_{2}(\X)} + \sup_{n\geq N}\bignorm{m(\theta_{n})}_{W^{\tau_{f}}_{2}(\X)}\right)\right),
\end{align*}
where $C = C_{0}\Lambda_{s,q}$ with $C_{0} =  C_{0}\left(\X,d,q,\tau_{f},\Theta_{N}^{*}\right)$, $h_{0} = h_{0}\left(R,\delta,d,\tau_{f},\Theta_{N}^{*}\right)$ and $\gamma = 2\vee q$.
\end{corollary}
We note that when the smoothness is well specified, the value of $\tau_{f}$ does not impact the decay rate of the right hand side since the rate will be slowed down by the $h_{X_n}^{d/\gamma -s}$ term which does not depend on $\tau_{f}$. This differs significantly from the noiseless case in Section \ref{sec:noiseless_convergence} where a large value of $\tau_{f}$ led to faster convergence rates, and demonstrates how a small amount of noise can significantly impact the convergence rate. 

Well-specified smoothness is a strong requirement. For the second corollary, we instead consider $\tau_k^+ =\tau_k^- = \tau$ for some $\tau\in\bbR$, but not necessarily $\tau_f=\tau$, when using quasi-uniform points and varying $\sigma_{n}$ according to the fill distance. Surprisingly, this is enough to obtain the same bound as Corollary \ref{cor:deterministic_corrupt}. 
\begin{corollary}\label{cor:adapting_sig}
	Fix $N\in\bbN$ and suppose Assumptions 1-\ref{assumption:double_SI_application} hold. Let $q\in[1,\infty]$, $s\in[0,(\tau_{f}\wedge\tau)^{*}]$, and $\tau_{k}^{+} = \tau_{k}^{-} = \tau$ and $\sigma_{n} = O(h_{X_{n}}^{\tau-d/2})$. Then, $\exists C,h_{0} > 0$ such that $\forall n\geq N$, $\forall X_{n}\subseteq\X$ quasi-uniform with $h_{X_{n}}\leq h_{0}$
	\begin{align*}
		&\mathbb{E}\left[\bignorm{f -R_{f,\sigma_{n}^{2},\varepsilon}^{m}(\theta_{n})}_{W^{s}_{q}(\X)}\right] \\
		&\qquad\leq Cn^{-\frac{1}{\gamma} +\frac{s}{d}}\left(\bbE[\norm{\varepsilon}_{2}] + n^{-\frac{(\tau_{f}\wedge\tau)}{d}+\frac{1}{2}}\bigg(\norm{f}_{W^{\tau_{f}}_{2}(\X)} + \sup_{n\geq N}\bignorm{m(\theta_{n})}_{W^{\tau_{f}}_{2}(\X)}\bigg)\right),
	\end{align*}
	where $C = C_{0}\Lambda_{s,q}$ with $C_{0} =  C_{0}\left(\X,d,q,\tau_{f},\Theta_{N}^{*}\right)$, $h_{0} = h_{0}\left(R,\delta,d,\tau_{f},\Theta_{N}^{*}\right)$ and $\gamma = 2\vee q$.
\end{corollary}
On top of using quasi-uniform points, this corollary suggests the following practical approach.

\begin{itemize}

\item \emph{Adaptive likelihood/Nugget}: When practitioners suspect that their likelihood might be misspecified, a sensible choice of nugget is $\sigma_n \propto h_{X_n}^{\tau-d/2}$. Interestingly, this is the same suggestion as for the case of Gaussian regression for noiseless data, which suggests that this choice may be sensible more broadly.
\end{itemize}

\subsubsection{Misspecified Interpolation Likelihood}

Our last main result considers arbitrary corruption when an interpolant is used, which is equivalent to taking $\sigma_{n} = 0$.

\begin{theorem}\label{thm:misspec_likelihood_inter}
	Fix $N\in\bbN$ and suppose Assumptions 1-\ref{assumption:double_SI_application} hold. Let $q\in[1,\infty]$ and $s\in[0,(\tau_{f}\wedge\tau^{-}_{k})^{*}]$. Then, $\exists  C,h_{0} > 0$ such that $\forall n\geq N$,  $\forall X_{n}\subseteq\X$ with $h_{X_{n}}\leq h_{0}$, when $\tau_{f}\geq\tau_{k}^{+}$
	\begin{align*}
		&\mathbb{E}\left[ \bignorm{f -R_{f,0,\varepsilon}^{m}(\theta_{n})}_{W^{s}_{q}(\X)} \right]  \\
		& \qquad\leq  Ch_{X_{n}}^{\frac{d}{\gamma}-s}\bigg[
		h_{X_{n}}^{\tau^{-}_{k}-\frac{d}{2}}\left(\norm{f}_{W^{\tau_{f}}_{2}(\X)} + \sup_{n\geq N}\bignorm{m(\theta_{n})}_{W^{\tau_{f}}_{2}(\X)}\right)
		 + \rho_{X_{n}}^{(\tau_{k}^{+}-\frac{d}{2})}\bbE[\norm{\varepsilon}_{2}]\bigg],
	\end{align*}
	and when $\tau_{f} < \tau_{k}^{+}$
	\begin{align*}
		&\mathbb{E}\left[ \bignorm{f -R_{f,\varepsilon}^{m}(\theta_{n})}_{W^{s}_{q}(\X)} \right]  \\
		& \leq  Ch_{X_{n}}^{\frac{d}{\gamma}-s}\rho_{X_{n}}^{(\tau_{k}^{+}-\tau_{f})}\bigg[
		h_{X_{n}}^{\left(\tau_{f}\wedge\tau^{-}_{k}\right)-\frac{d}{2}}\left(\norm{f}_{W^{\tau_{f}}_{2}(\X)} + \sup_{n\geq N}\bignorm{m(\theta_{n})}_{W^{\tau_{f}}_{2}(\X)}\right)+ \rho_{X_{n}}^{(\tau_{f}-\frac{d}{2})}\bbE[\norm{\varepsilon}_{2}]\bigg],
	\end{align*}
	where $C = C_{0}\Lambda_{s,q}$ with $C_{0} =  C_{0}\left(\X,d,q,\tau_{f},\Theta_{N}^{*}\right)$, $h_{0} = h_{0}\left(R,\delta,d,\tau_{f},\Theta_{N}^{*}\right)$ and $\gamma = 2\vee q$.
\end{theorem}

If $\varepsilon = 0$, then there is no noise and we recover the well-specified likelihood result for interpolation from Theorem \ref{thm:Misspecified_Sobolev_Interpolation}. We now study the impact of model choice.

\begin{itemize}
\item \emph{Experimental design:} In this bound, there is a $\rho_{X_n}$ term multiplied by $\bbE[\norm{\varepsilon}_{2}]$ whose exponent does not vanish when the smoothness is well specified. This is in contrast to Theorem \ref{thm:Misspecified_Sobolev_Det_Regression} where the exponent of the $\rho_{X_n}$ term vanished when the smoothness was well specified, and $\rho_{X_n}$ did not interact with $\bbE[\norm{\varepsilon}_{2}]$. This can be interpreted as $R^{m}_{f,\varepsilon}$ being less stable than $R^{m}_{f,\sigma_{n}^{2},\varepsilon}$ with respect to noise and point placement, and suggests that the use of quasi-uniform point is strongly recommended, even when the smoothness is well-specified.
\end{itemize}

The following corollary shows the same bound as Corollary \ref{cor:deterministic_corrupt} and Corollary \ref{cor:adapting_sig} can be obtained without the assumption of fixed kernel smoothness as long as quasi-uniform points are used. 
\begin{corollary}
	Fix $N\in\bbN$ and suppose Assumptions 1-\ref{assumption:double_SI_application} hold. Let $q\in[1,\infty]$ and $s\in[0,(\tau_{f}\wedge\tau^{-}_{k})^{*}]$. Then, $\exists C,h_{0} > 0$ such that $\forall n\geq N$, $\forall X_{n}\subseteq\X$ quasi-uniform with $h_{X_{n}}\leq h_{0}$
	\begin{align*}
		&\mathbb{E}\left[\bignorm{f -R_{f,0,\varepsilon}^{m}(\theta_{n})}_{W^{s}_{q}(\X)}\right] \\
		&\qquad\leq Cn^{-\frac{1}{\gamma} +\frac{s}{d}}\left(\bbE[\norm{\varepsilon}_{2}] + n^{-\frac{(\tau_{f}\wedge\tau_{k}^{-})}{d}+\frac{1}{2}}\left(\norm{f}_{W^{\tau_{f}}_{2}(\X)} + \sup_{n\geq N}\bignorm{m(\theta_{n})}_{W^{\tau}_{2}(\X)}\right)\right),
	\end{align*}
	where $C = C_{0}\Lambda_{s,q}$ with $C_{0} =  C_{0}\left(\X,d,q,\tau_{f},\Theta_{N}^{*}\right)$, $h_{0} = h_{0}\left(R,\delta,d,\tau_{f},\Theta_{N}^{*}\right)$ and $\gamma = 2\vee q$.
\end{corollary}
Compared to Corollary \ref{cor:deterministic_corrupt} the requirement of quasi-uniform points is stronger than just $h_{X_{n}}\leq Cn^{-\frac{1}{d}}$, but this allows us to weaken the assumptions on the smoothness of the kernel. Indeed, as opposed to Corollary \ref{cor:adapting_sig}, the kernel smoothness is allowed to alter with $n$. However, $\sigma_{n} = 0$ means the approximation is harder to compute due to the matrix inversion being less stable.

\section{Implications for Bayesian Numerical Methods}\label{sec:bayesian_pn}

We demonstrate the applicability of our theorems to Bayesian probabilistic numerical methods, specifically Bayesian quadrature and Bayesian optimisation. These methods use GP approximations to solve numerical tasks, and can therefore inherit some of the convergence guarantees presented in the previous section.

\subsection{Bayesian Quadrature}

In Bayesian quadrature (BQ), the goal is to approximate some integral $\int_{\mathcal{X}} f(x) p(x) dx$. To do so, a GP prior is placed on $f$. This is conditioned on function evaluations to obtain a posterior on $f$, which itself implies a Gaussian posterior on the value of the integral. The posterior mean on this integral is used as an estimate of the integral, see e.g.\ \citep{Briol2019} and the accompanying discussion for an in-depth overview. The most up-to-date convergence guarantees are available from \citet{Kanagawa2019}. These consider the problem of smoothness misspecification in the interpolation setting.

We now highlight how the results of this paper can refine theory for BQ, but also lead to results in settings with likelihoods which have not yet been considered. First we consider interpolation, the proof is a combination of Theorem \ref{thm:Misspecified_Sobolev_Interpolation} with $q = 1, s = 0$ and H\"older's inequality.

\begin{theorem}\label{thm:BQ}
	Fix $N\in\bbN$ suppose Assumptions 1-4 hold. Then $\exists C_{0},h_{0} > 0$ such that $\forall n\geq N$, $\forall X_{n}\subseteq\X$ with $h_{X_{n}}\leq h_{0}$ and $\forall p\in L^{2}(\X)$ 
	\begin{align*}
		&\left\lvert \int_{\X}f(x)p(x)dx - \int_{\X}R^{m}_{f}(\theta_{n})(x)p(x)dx\right\rvert\\
		&\qquad\qquad \leq C\norm{p}_{L^{2}(\X)}h_{X_n}^{(\tau_{f}\wedge\tau^{-}_{k})}\rho_{X_n}^{(\tau^{+}_{k}-\tau_{f})_{+}}\left(\norm{f}_{W^{\tau_{f}}_{2}(\X)} + \sup_{n\geq N}\bignorm{m(\theta_{n})}_{W^{\tau_{f}}_{2}(\X)}\right),
	\end{align*}
where $C = C(\X, d,\tau_{f},\Theta^{*}_{N})$,  $h_{0} = h_{0}(R,\delta,d,\tau_{f},\Theta^{*}_{N})$.
\end{theorem}

This result generalizes \citep[Theorem 3]{Kanagawa2019} by allowing a greater range of values for $\tau_k^-, \tau_k^+$ and $\tau_f$. It also takes into account the adaptation of hyperparameters with $n$, which has not been considered in the literature. Next, we consider a correctly specified Gaussian likelihood. The proof is a combination of Corollary \ref{cor:gaussian_noise_cor} with $q=1, s=0$ and H\"older's inequality.
 
\begin{theorem}
	Fix $N\in\bbN$ and suppose Assumptions 1-\ref{ass:target_extension} hold. Let $\tau_{k}^{+} = \tau_{k}^{-} = \tau_{f} + d/2$. Then, $\exists C,h_{0} > 0$ such that $\forall n\geq N$, $\forall X_{n}\subseteq\X$ quasi-uniform with $h_{X_{n}}\leq h_{0}$ and $\forall p\in L^{2}(\X)$ 
	\begin{align*}
		& \mathbb{E}\left[\left\lvert \int_{\X}f(x)p(x)dx - \int_{\X}R^m_{f,\sigma^2,\epsilon}(\theta_{n})(x)p(x)dx\right\rvert\right] \; \leq \; C\norm{p}_{L^{2}(\X)} n^{-\frac{\tau_{f}}{2\tau_{f} + d}},
	\end{align*}
where $C=C\left(\X,d,\tau_{f},\norm{f}_{W^{\tau_{f}}_{2}(\X)},\sup_{n\geq N}\norm{m(\theta_{n})}_{W^{\tau_{f}}_{2}(\X)},\Theta_{N}^{*}\right)$, $h_{0} = h_{0}\left(R,\delta,d,\tau_{f},\Theta^{*}_{N}\right)$.
\end{theorem}
This result provides the very first bound for BQ with a correctly specified Gaussian likelihood. This may be particularly useful for applications of BQ in inverse problems and computer models, where the integrand cannot be evaluated exactly. 

The two results above are illustrations of bounds that can be obtained using the theory in our paper. However, it would be straightforward to obtain results in other settings, including misspecified smoothness or misspecified likelihoods, using the same proof technique with some of the other bounds in Section \ref{sec:theory}. All of the previous recommendation on model choice are also appropriate for BQ, with the exception of the experimental design, for which it is recommended to use quasi-uniform points which concentrate in areas where $p$ is large. 

\subsection{Bayesian Optimization}

In Bayesian optimization (BO), the goal is to maximise some unknown function. This is done using a GP surrogate, and points are usually chosen using an acquisition function which balances exploration and exploitation of the GP model given the observed data up to that iteration. Common examples include the Upper Confidence Bound and Expected Improvement acquisition functions \citep{Shahriari2016}. In the noiseless case, SDA results were employed by \citet{Bull2011} and a modification was proposed to the standard expected improvement acquisition function, to ensure greater coverage of the domain\footnote{It is important to note that the definition of ``quasi-uniformity'' by \citet{Bull2011} is strictly weaker than the standard definition in SDA. It only requires $h_{X_{n}}\leq Cn^{-1/d}$, which is implied by standard definition used in this paper.}. 

Existing theoretical work on Bayesian optimization that establishes convergence under various acquisition functions do not accommodate for misspecification of functions smoothness \citep{Bull2011,Srinivas2010,Vazquez2010}. This problem is addressed by \citet{Berkenkamp2019} using a hyperparameter alteration regime which enlarges the RKHS until the target function is contained in it. Motivated by the content of Theorem \ref{thm:Misspecified_Sobolev_Interpolation}, we investigate a different approach to tackle smoothness misspecification, relying on a modification of existing acquisition functions to promote quasi-uniform points and then employing the proof technique by \citet{Bull2011}.

The $\gamma$-stabilized algorithm framework \citep{Wenzel2019} facilitates such a modification. For any acquisition function $F\colon\X\rightarrow\bbR$, kernel $k$  and $\gamma\in(0,1]$, the $(n+1)$-th step consists of picking $x_{n+1} = \sup_{x\in\X_{n,\gamma}}F(x)$ where $\X_{n,\gamma} = \{x\in\X\colon P_{n}(x)\geq \gamma\norm{P_{n}}_{L^{\infty}}\}$, $P_{n}(x) = \bar{k}(x,x)^{\frac{1}{2}}$ and $\bar{k}$ is the posterior variance after observing the first $n$ points, see Section \ref{sec:background}. Such point selection encourages exploration since it only allows points to be picked from areas of non-trivial variance. If $k$ is translation invariant and $\tau$-smooth with $\tau > d/2+1$ then the resulting point set is quasi-uniform \citep[Theorem 14, Theorem 18]{Wenzel2019}. This is a modification to the standard BO procedure of picking $x_{n+1}$ as the maximum of $F$ over all of $\X$. 

For $n\in\bbN$ and any acquisition function $F$ define the $(\gamma,F,n)$ strategy as picking $x_1$ arbitrarily, then points $\{x_{i}\}_{i=2}^{n-1}$ according to the $\gamma$-stabilized $F$, and $x_{n}$ as the maximum of $R_{f}$, the kernel interpolant of $f$ based on $\{x_{i}\}_{i=1}^{n-1}$. The next result gives a bound for the performance of this strategy. 

\begin{theorem}\label{thm:BO}
	Suppose Assumptions 1 \& \ref{assumption:mean} hold, $k$ is a $\tau$-smooth translation invariant kernel with $\tau > d/2 +1$. Then, $\exists n_{0}\in\bbN$ such that if $n\geq n_{0}$, the $(\gamma,F,n)$-strategy satisfies:
	\begin{align*}
		\lvert \argmax_{x \in \mathcal{X}} f(x)-f(x_{n})\rvert= Cn^{-\frac{(\tau\wedge\tau_{f})}{d}+\frac{1}{2}},
	\end{align*}
	where $C = C(\X, d,\tau_{f})$ and $n_{0} = n_{0}(R,\delta,d,\tau_{f})$.
\end{theorem}
In terms of worst-case error, in which the slowest rate is considered over the unit ball of the RKHS, this is the best possible rate given the smoothness of the target function and kernel, as shown by \citet[Theorem 1]{Bull2011}. The $1/2$ appears in our bound due to a different parameterisation of kernel smoothness than \citet{Bull2011}.  However, Theorem \ref{thm:BO} is more general than the result by \citet[Theorem 1]{Bull2011} since it applies to functions outside the RKHS of $k$. This is the first BO strategy which achieves the optimal rate in the case of smoothness misspecification.  

Once again, we conclude by noting that this theorem is only an illustration of the implications of our results on convergence guarantees for GPs to the BO setting, and many other cases could be considered including likelihood misspecification.

\section{Conclusions}\label{sec:conclusions}

In this paper, we have presented novel error bounds for GP means under misspecified likelihoods and smoothness, expressed in terms of observation error, point placement and choice of GP model. Our results apply under four different observation models. Where the assumption of no noise is correct, where the assumption of no noise is incorrect, where the assumption of a Gaussian likelihood is correct and when the assumption of a Gaussian likelihood is incorrect. In each setting, our results demonstrate the impact of the choice of hyperparameters and the experimental design. As such, our results can guide practitioners who need to select a specific GP algorithm, by allowing them to tailor this choice to the application at hand. 

The bounds offer improvements over existing results which we have highlighted. Applications to Bayesian numerical methods were presented such as the first error bounds for BQ with deterministic point selection and Gaussian observation noise and BO with misspecified smoothness. In both instances the use of point picking strategies which produce quasi-uniform points, as opposed to specific hyperparameter selection methods, are of critical importance. 

We believe there are many more opportunities to combine GP and SDA methods. For example dealing with smoothness and likelihood misspecification when the approximating function is infinitely smooth, such as when a Gaussian kernel is used for approximation a common choice in practice. Additionally, analogies of the results in this paper for functions with vector valued output or structured output, such as additive functions, would be an important avenue of research and would follow naturally from the insights that SDA offers for such scenarios.


\acks{We would like to thank Toni Karvonen, Andrew Duncan, three anonymous reviewers and the editor for helpful comments when writing this paper. GW was supported by an EPSRC Industrial CASE award [EP/S513635/1] in partnership with Shell UK Ltd. FXB was supported by an Amazon Research Award on ``Transfer Learning for Numerical Integration in Expensive Machine Learning Systems''. MG was supported by the EPSRC grants [EP/T000414/1, EP/R018413/2, EP/P020720/2, EP/R034710/1, EP/R004889/1]. FXB and MG were also supported by the Lloyds Register Foundation Programme on Data-Centric Engineering and The Alan Turing Institute under the EPSRC grant [EP/N510129/1].}


\appendix


\section{The Design Region} \label{sec:assumptions_domain}
In this first appendix, we briefly recall common terminology from the literature on scattered data approximation which is used throughout the paper. 

A domain shall mean an open connected set in $\mathbb{R}^{d}$. A domain satisfies the $(R,\delta)$ interior cone condition if for $R > 0$ and angle $\delta\in (0,\pi /2)$ we have that $ \forall x\in\X$,  $\exists\:\xi(x)$ such that the cone 
\begin{align*}
  C\left(x,\xi(x),\delta,R\right) = \left\{x + \lambda y\::\: y\in\mathbb{R}^{d}, \norm{y}_{2} = 1, y^{\top}\xi(x)\geq \cos(\delta), \lambda\in [0,R]\right\},
\end{align*}
is contained in $\X$. An open set $\X_i\subseteq\mathbb{R}^{d}$ is called a special Lipschitz domain \citep[Page 181]{Stein1970} if there exists a rotation of $\X_i$, denoted by $\tilde{\X}_i$, and a function $\psi : \mathbb{R}^{d-1}\xrightarrow{}\mathbb{R}$ which satisfies the following
\begin{enumerate}
    \item $\tilde{\X}_i = \{(x,y)\in\mathbb{R}^{d}\:\: y > \psi(x)\}$,
    \item $\psi$ is a Lipschitz function such that $\lvert \psi(x) - \psi(x')\rvert\leq M\norm{x-x'}_{2}$  $\forall x,x'\in\mathbb{R}^{d-1}$ where $M > 0$.
\end{enumerate}

Consider a domain $\X \subseteq \mathbb{R}^d$ and denote its boundary by $\partial \X$. We say $\partial \X$ is a Lipschitz boundary \citep[Page 189]{Stein1970} if $\exists\: \varepsilon > 0$, $N\in\mathbb{N}$, $M > 0$, and open sets $U_{1}, U_{2},\ldots,U_{L}\subset\mathbb{R}^{d}$, where $L\in\mathbb{N}\cup\{\infty\}$, such that the following conditions are satisfied
\begin{enumerate}
    \item For any $x\in\partial\X$, there exists an index $i$ such that $B(x,\varepsilon)\subset U_{i}$,
    \item $U_{i_{1}}\cap\cdots\cap U_{i_{N+1}} = \emptyset$ for any distinct indices $\{i_{1},\ldots i_{N+1}\}$,
    \item For each index $i$ there exists a special Lipschitz domain $\X_{i}\subset\mathbb{R}^{d}$ with Lipschitz bound $b$ such that $b\leq M$ and $U_{i}\cap\X = U_{i}\cap\X_{i}$. 
\end{enumerate}
and we call any bounded domain satisfying the $(R,\delta)$ interior cone condition with a Lipschitz boundary a $\mathcal{L}(R,\delta)$-domain.

\section{Preliminary Results}

This section covers results to be used throughout the rest of the proofs, namely a sampling inequality, restriction and extension of functions in RKHS and the Pythagorian property.

Sampling inequalities \citep{Narcowich2006,Rieger2010,Arcangeli2011,Arcangeli2014} are powerful inequalities for functions in Sobolev spaces which facilitate the systemisation of approximation error bounds. The result below is a special case of the result by \citet[Theorem 3.2]{Arcangeli2011} where the integrability parameter in the right hand side Sobolev norm is set to two and so is the parameter $p$ of the $l_p$ norm. 

\begin{theorem}\label{thm:Arcangeli_SI}
	Let $\X$ be a $\mathcal{L}(R,\delta)$-domain, $\tau > d/2$ and $q\in[1,\infty]$. Then, $\exists C,h_{0} > 0$ such that $\forall X\subseteq\X$ with $h_{X}\leq h_{0}$, any $f\in W^{\tau}_{2}(\X)$ and any $s\in[0,\tau^{*}]$
	\begin{align*}
		\norm{f}_{W^{s}_{q}(\X)}\leq C\Lambda_{s,q}\left(h_{X}^{\tau - s - d\left(\frac{1}{2}-\frac{1}{q}\right)_{+}}\norm{f}_{W^{\tau}_{2}(\X)} + h_{X}^{\frac{d}{\gamma}-s}\norm{f_{X}}_{2}\right),	
	\end{align*}
	where $C= C(\X,d,\tau,q)$, $h_{0} = h_{0}(\delta, R,d,\tau)$ and $\gamma = 2\vee q$.
\end{theorem}
Discussion of how the domain, smoothness of the function and point set affect the constants is provided by \citet{Arcangeli2011}. It is important to note that the dependence on $\tau$ in $h_{0}$ is only through $\lfloor\tau\rfloor$, this can be seen from inspection of the proof. The sampling inequality above is defined for norms over $\X$, but our proofs will be based on Fourier transforms which will be defined for functions over $\bbR^{d}$ therefore results facilitating the restriction and extension of functions between $\X$ and $\bbR^{d}$ are required. To this end the Sobolev extension theorem is required, stated below. 

\begin{theorem}\label{thm:Sobolev_Extension}
	Let $\mathcal{X}\subseteq\mathbb{R}^{d}$ be a bounded Lipschitz domain, $\tau\geq 0$ and $p\in[1,\infty)$. There exists an extension map $\mathcal{E}\colon W^{\tau}_{p}(\mathcal{X})\to W^{\tau}_{p}(\mathbb{R}^{d})$ such that $\forall f\in W^{\tau}_{p}(\mathcal{X})$ we have $\restr{\mathcal{E}f}{\mathcal{X}} = \restr{f}{\mathcal{X}}$ and $\norm{\mathcal{E}f}_{W^{\tau}_{p}(\mathbb{R}^{d})} \leq C\norm{f}_{W^{\tau}_{p}(\mathcal{X})}$ where $C = C(\mathcal{X},d,\tau,p) > 0$ is a constant independent of $f$.
\end{theorem}

The Sobolev extension theorem is used by \citet[Corollary 10.48]{Wendland2005} to ensure that, along with some assumptions on $\X$ satisfied by Assumption \ref{assumption:domain}, if $\calH_{k}(\bbR^{d})$ is norm equivalent to $W^{\tau}_{2}(\bbR^{d})$ then $\calH_{k}(\X)$ is norm equivalent to $W^{\tau}_{2}(\X)$. Finally, the next two lemmas assure us that the minimal norm properties of the kernel interpolant and kernel regression function still hold along with the Pythagorean property for kernel interpolant. For a proof, see e.g.\ \citep[Corollary 10.25]{Wendland2005}.

\begin{lemma}
	Let $\mathcal{X}\subseteq\mathbb{R}^{d}$, $X\subseteq\mathcal{X}$ a finite subset, $k$ a kernel over $\bbR^{d}\times\bbR^{d}$ and $f\in \mathcal{H}_{k}(\mathcal{X})$ then
\begin{align*}
    \restr{R_{f}}{\mathcal{X}} &  = \argmin_{\substack{g\in \mathcal{H}_{k}(\mathcal{X}) \\ g_{X} = f_{X}}}\norm{g}_{\mathcal{H}_{k}(\mathcal{X})} \qquad
    \restr{R_{f,n\lambda,\varepsilon}}{\mathcal{X}}  = \argmin_{g\in \mathcal{H}_{k}(\mathcal{X})}S(g,\lambda,\mathcal{X}),
\end{align*}
where $S(g,\lambda,\mathcal{X})$ is the regularized least squares problem defined in Section \ref{sec:background}. 
\end{lemma}

\begin{proof}
 The case of $\X = \bbR^{d}$ is obtained by standard arguments \citep[Theorem 3.4, Theorem 3.5]{Kanagawa2018Review} so we restrict to the case when $\X$ is a strict subset of $\bbR^{d}$. We shall only prove the first statement since the second proof is analogous. The interpolant restricted to $\mathcal{X}$ equals $f$ on $X$ since $X\subseteq\mathcal{X}$ and by definition $\restr{R_{f}}{\mathcal{X}}\in \mathcal{H}_{k}(\mathcal{X})$, therefore 
\begin{align*}
	\bignorm{\restr{R_{f}}{\mathcal{X}}}_{\mathcal{H}_{k}(\mathcal{X})} \geq \min_{\substack{g\in \mathcal{H}_{k}(\mathcal{X}) \\ \restr{g}X = \restr{f}{X}}}\norm{g}_{\mathcal{H}_{k}(\mathcal{X})}.
\end{align*}
The rest of the proof will be done by contradiction. Suppose $\exists g\in \mathcal{H}_{k}(\mathcal{X})$ such that $\restr{g}{X} = \restr{f}{X}$ and $\norm{g}_{\mathcal{H}_{k}(\mathcal{X})} < \bignorm{\restr{R_{f}}{\mathcal{X}}}_{\mathcal{H}_{k}(\mathcal{X})}$. Then, by definition of the norm on $\mathcal{H}_{k}(\mathcal{X})$
\begin{align*}
	\norm{g}_{\mathcal{H}_{k}(\mathcal{X})} = \inf_{\substack{h\in \mathcal{H}_{k}(\bbR^d)\\ \restr{h}{\mathcal{X}} = \restr{g}{\mathcal{X}}}}\norm{h}_{\mathcal{H}_{k}(\bbR^d)} < \bignorm{\restr{R_{f}}{\mathcal{X}}}_{\mathcal{H}_{k}(\mathcal{X})} \leq \bignorm{R_{f}}_{\mathcal{H}_{k}(\bbR^d)}.
\end{align*}
By definition of the infimum, $\exists h\in \mathcal{H}_{k}$ such that $\restr{h}{\mathcal{X}} = \restr{g}{\mathcal{X}}$ and $\norm{h}_{\mathcal{H}_{k}(\bbR^d)} < \norm{R_{f}}_{\mathcal{H}_{k}(\bbR^d)}$. But $X\subseteq\mathcal{X}$ hence $h_{X} = g_{X} = f_{X}$ which contradicts norm minimality of $R_{f}$ over $\bbR^{d}$. This completes the proof.
\end{proof}

\begin{lemma}\label{lem:pythag}
	Let $\mathcal{X}\subseteq\mathbb{R}^{d}$, $k$ a kernel over $\X\times\X$ and $f\in \mathcal{H}_{k}(\mathcal{X})$ then we have the Pythagorean property for the interpolant: $\bignorm{f-R_{f}}_{\mathcal{H}_{k(\X)}}^{2} + \bignorm{R_{f}}_{\mathcal{H}_{k(\X)}}^{2} = \norm{f}_{\mathcal{H}_{k(\X)}}^{2}$.
\end{lemma}

\section{Proof of Theorem \ref{thm:Misspecified_Sobolev_Interpolation}}
A key intermediate result is a slight generalisation of the sampling inequality by \citet[Theorem 4.2]{Narcowich2006} which facilitates bounds for the misspecified smoothness scenario. 

\begin{theorem}\label{thm:misspecified_inequality}
	Suppose $\X$ is a $\mathcal{L}(R,\delta)$-domain and $k$ is $\gamma$-smooth for $\gamma > d/2$. Then, $\exists C,h_{0} > 0$ such that $\forall X\subseteq\X$ with $h_{X}\leq h_{0}$, we have $\forall f\in W^{\tau}_{2}(\X)\:\forall \mu\in[0,\tau]$
	\begin{align*}
		\bignorm{f-R_{f}}_{W^{\mu}_{2}(\X)} \leq Ch_{X}^{\tau-\mu}\rho_{X}^{\gamma-\tau}\norm{f}_{W^{\tau}_{2}(\X)},
	\end{align*}
	where $C= C(\X,d,\tau)$, $h_{0} = h_{0}(\delta, R,d,\tau)$.
\end{theorem} 
\begin{proof}
The proof is identical to the proof by \citet[Theorem 4.2]{Narcowich2006}, but with different assumptions on $\gamma$ and $\X$. Specifically the proof by \citet[Theorem 4.2]{Narcowich2006} uses the result by \citet[Lemma 4.1]{Narcowich2006} for which a strictly smaller range of $\gamma$ is permitted. However Theorem \ref{thm:Arcangeli_SI} generalises the older bound by \citet[Lemma 4.1]{Narcowich2006} and simply requires $\gamma > d/2$. Additionally compactness of $\X$ was assumed by \citet[Theorem 4.2]{Narcowich2006} to use a version of the Sobolev extension theorem but Theorem \ref{thm:Sobolev_Extension} can instead be used to obtain the same conclusion for $\mathcal{L}(R,\delta)$-domains. 
\end{proof}

We begin by expressing the error for the interpolant $R_{f}^{m}(\theta_{n})$ in terms of two zero-mean GP interpolation problems
\begin{align}
	\bignorm{f-R_{f}^{m}(\theta_{n})}_{W^{s}_{q}(\X)} 
	& = \bignorm{f-R_{f}(\theta_{n}) - m(\theta_{n}) +R_{m(\theta_{n})}(\theta_{n})}_{W^{s}_{q}(\X)} \nonumber\\
	&\leq \bignorm{f-R_{f}(\theta_{n})}_{W^{s}_{q}(\X)} + \bignorm{m(\theta_{n})-R_{m(\theta_{n})}(\theta_{n})}_{W^{s}_{q}(\X)}. \label{eq:start_interpolation_thm}
\end{align}
The equality follows by the definition in \eqref{eq:def_approximation} and the inequality is the triangle inequality. Therefore zero-mean GP interpolation problems only needs to be dealt with. An upper-bound on the first term naturally leads to an upper bound on the second since Assumption \ref{assumption:mean} imposes that $m(\theta_{n})$ is at least as smooth as the target function. For $n\geq N$ and $s\in \left[0,(\tau_{f}\wedge\tau^{-}_{N})^{*}\right]$, applying Theorem \ref{thm:Arcangeli_SI} to the function $f-R_{f}(\theta_{n})$ over all smoothness levels $\{\tau_{f}\wedge\tau(\theta_{n})\}_{n\geq N}$ yields
\begin{align}\label{SI_on_difference_interpolant}
	\bignorm{f-R_{f}(\theta_{n})}_{W^{s}_{q}(\X)}
	& \leq C_{1}\Lambda_{s,q}h_{X_{n}}^{\left(\tau_{f}\wedge\tau(\theta_{n})\right)-s-d\left(\frac{1}{2}-\frac{1}{q}\right)_{+}}\bignorm{f-R_{f}(\theta_{n})}_{W^{\tau_{f}\wedge\tau(\theta_{n})}_{2}(\X)},
\end{align}
for $h_{X_{n}}\leq h_{1}$ where $C_{1} = C_{1}(\X,d,\tau_{f},q,\Theta_{N}^{*})$ and $h_{1} = h_{1}(R,\delta,d,\tau_{f},\Theta_{N}^{*})$ are respectively the supremum and infimum over $n\geq N$ of the constants obtained from applying Theorem \ref{thm:Arcangeli_SI} with smoothness parameter $\tau_{f}\wedge\tau(\theta_{n})$. Due to Assumption \ref{assumption:smoothness}, $\tau_{f}\wedge\tau(\theta)$ takes finitely many values so the infimum and supremum are over a finite number of values. This immediately gives $C_{1} < \infty$ and $h_{1} > 0$ and the same logic will be employed whenever Theorem \ref{thm:Arcangeli_SI} is used again.  The residual terms are zero since $R_{f}(\theta_{n})$ interpolates $f$ at the observation points. 

For the case $\tau_{f} \geq \tau(\theta_{n})$, the target function $f$ is in the RKHS of $k(\theta_{n})$ so we can derive the following inequality
\begin{align}
 	\bignorm{f-R_{f}(\theta_{n})}_{W^{\tau_{f}\wedge\tau(\theta_{n})}_{2}(\X)} 
 	& = \bignorm{f-R_{f}(\theta_{n})}_{W^{\tau(\theta_{n})}_{2}(\X)} \nonumber\\
 	 & \leq C_{u}(\theta_{n})\norm{f-R_{f}(\theta_{n})}_{\mathcal{H}_{k(\theta_{n})}(\X)} \label{eq:larger_tau_result_1} \\
 	&\leq C_{u}(\theta_{n})\norm{f}_{\mathcal{H}_{k(\theta_{n})}(\X)} \label{eq:larger_tau_result_2}\\
 	& \leq C_{u}(\theta_{n})C_{l}(\theta_{n})^{-1}\norm{f}_{W^{\tau(\theta_{n})}_{2}(\X)} \label{eq:larger_tau_result_3}\\
 	& \leq C_{N}\norm{f}_{W^{\tau_{f}}_{2}(\X)}\label{eq:larger_tau_result_4}. 
\end{align}

The inequalities in \eqref{eq:larger_tau_result_1} and \eqref{eq:larger_tau_result_3} follow from the norm equivalence between the RKHSs and Sobolev spaces with constants given in \eqref{eq:norm_equivalence_constants}. The inequality in \eqref{eq:larger_tau_result_2} is due to the Pythagorean property in Lemma \ref{lem:pythag}, \eqref{eq:larger_tau_result_4} is obtained by upper bounding by the largest constants over all values of $\{\theta_n\}_{n \geq N}$, which can be done by Assumption \ref{assumption:parameters}, and the fact that the $\norm{\cdot}_{W^{\tau_{f}}_{2}(\X)}$ norm which is larger than the $\norm{\cdot}_{W^{\tau(\theta_{n})}_{2}(\X)}$ norm since we are currently dealing with the case $\tau_{f}\geq\tau(\theta_{n})$. 

For the case $\tau(\theta_{n}) > \tau_{f}$, setting $\gamma = \tau(\theta_{n})$ and $\mu = \tau_{f}$ in Theorem \ref{thm:misspecified_inequality} gives
\begin{align}\label{eq:smaller_tau}
\bignorm{f-R_{f}(\theta_{n})}_{W^{\tau_{f}}_{2}(\X)} \leq C_{2}\rho_{X_{n}}^{\tau(\theta_{n})-\tau_{f}}\norm{f}_{W^{\tau_{f}}_{2}(\X)},
\end{align} 
for $h_{X}\leq h_{2}$ where $C_{2} = C_{2}(\X,d,\tau_{f},\Theta_{N}^{*})$ and $h_{2} = h_{2}(R,\delta,d,\tau_{f},\Theta_{N}^{*})$. By the same reasoning as the discussion after \eqref{SI_on_difference_interpolant} $h_{2} > 0$ and $C_{2} < \infty$. Now combine \eqref{SI_on_difference_interpolant}, \eqref{eq:larger_tau_result_4} and \eqref{eq:smaller_tau}
\begin{align}
& \bignorm{f-R_{f}(\theta_{n})}_{W^{s}_{q}(\X)} \nonumber \\
& \qquad \leq 
\begin{cases}
C_{1}C_{N} \Lambda_{s,q} h_{X}^{(\tau_{f}\wedge\tau(\theta_{n}))-s-d\left(\frac{1}{2}-\frac{1}{q}\right)_{+}}  \norm{f}_{W^{\tau_{f}}_{2}(\X)}& \text{ if } \tau_{f} < \tau(\theta_{n}) \\
C_{1}C_{2} \Lambda_{s,q} h_{X}^{(\tau_{f}\wedge\tau(\theta_{n}))-s-d\left(\frac{1}{2}-\frac{1}{q}\right)_{+}}\rho_{X}^{\tau(\theta_{n})-\tau_{f}} \norm{f}_{W^{\tau_{f}}_{2}(\X)}& \text{ if } \tau_{f} \geq \tau(\theta_{n})
\end{cases} \nonumber \\
& \qquad \leq C_3 \Lambda_{s,q} h_{X}^{(\tau_{f}\wedge\tau_{k}^{-})-s-d\left(\frac{1}{2}-\frac{1}{q}\right)_{+}}\rho_{X}^{(\tau_{k}^{+}-\tau_{f})_+} \norm{f}_{W^{\tau_{f}}_{2}(\X)}, \label{eq:sup_n_interpolant}
\end{align}
where the inequality in \eqref{eq:sup_n_interpolant} is obtained by taking the largest bound over parameter values $\{\theta_n\}_{n \geq N}$ and $C_3 = \max(C_1 C_{N}, C_1 C_2)$. To conclude the proof apply the upper bound in \eqref{eq:sup_n_interpolant} to each term of \eqref{eq:start_interpolation_thm} then set $C_{0}$ to be two times the maximum of the constants for each term and $h_{0}$ the minimum of the fill distance constants related to each term.

\section{Proof of Theorem \ref{thm:Misspecified_Sobolev_Det_Regression}}\label{appendix:deterministic_noise}

To obtain a bound in the scenario of corrupted data, we cannot use Theorem \ref{thm:misspecified_inequality} or Lemma \ref{lem:pythag} since they only apply to interpolants. Instead, Theorem \ref{thm:Misspecified_Sobolev_Det_Regression} will follow from Theorem \ref{thm:Misspecified_Sobolev_corruption} and Lemma \ref{lem:Regression_Best_Approximation} along with the band-limited function techniques pioneered by \citet{Narcowich2006}.

\begin{lemma}\label{lem:Regression_Best_Approximation}
Let $k$ be a kernel on $\X\times\X$, $f\in \mathcal{H}_{k}(\X)$, $\sigma > 0$ and assume observations $y_{i} = f(x_{i}) + \varepsilon_{i}$ at $X= \{x_{i}\}_{i=1}^{n}$ for some $\varepsilon\in\mathbb{R}^{n}$ then
\begin{align*}
	\bignorm{R_{f,\sigma^{2},\varepsilon}}_{\mathcal{H}_{k}(\X)} & \leq \left(\sigma^{-2}\norm{\varepsilon}^{2}_{2} + \norm{f}_{\mathcal{H}_{k}(\X)}^{2}\right)^{\frac{1}{2}} \\
	\bignorm{(f-R_{f,\sigma^2,\varepsilon})_{X}}_{2} & \leq\norm{\varepsilon}_{2} + \left(\norm{\varepsilon}^{2}_{2} + \sigma^2 \norm{f}_{\mathcal{H}_{k}(\X)}^{2}\right)^{\frac{1}{2}}.
\end{align*}
\end{lemma}

\begin{proof}
By triangle inequality 
\begin{align*}
	\bignorm{(f-R_{f,\sigma^2,\varepsilon})_{X}}_{2} = \bignorm{(y-\varepsilon-R_{f,\sigma^2,\varepsilon})_{X}}_{2} \leq \bignorm{(y-R_{f,\sigma^2,\varepsilon})_{X}}_{2} + \norm{\varepsilon}_{2}. 
\end{align*}
Combining this with  the inequality below completes the proof
\begin{align}
	& \max\left(\bignorm{\left(y-R_{f,\sigma^2,\varepsilon}\right)_{X}}^{2}_{2}
	 \; , \; \sigma^2 \bignorm{R_{f,\sigma^2,\varepsilon}}_{\mathcal{H}_{k}(\X)}^{2}\right)\nonumber \\
	& \qquad \leq n S(R_{f,\sigma^2,\varepsilon},\sigma^2 n^{-1}, \mathcal{X}) \label{eq:proof_residual_bound_1}\\
	& \qquad \leq n S(f,\sigma^2 n^{-1}, \mathcal{X}) 
	 = \norm{\varepsilon}^{2}_{2} + \sigma^2\norm{f}_{\mathcal{H}_{k}(\X)}^{2}. \label{eq:proof_residual_bound_2}
\end{align}
Where \eqref{eq:proof_residual_bound_1} uses the definition of the optimisation problem $S$, see Section \ref{sec:background}, and \eqref{eq:proof_residual_bound_2} follows since $R_{f,\sigma^2,\varepsilon}$ solves the optimisation problem $S$.
\end{proof}

\begin{theorem}\label{thm:Misspecified_Sobolev_corruption}Fix $N\in\bbN$ suppose Assumptions 1-\ref{assumption:double_SI_application} hold and each observation is corrupted by some $\varepsilon_{i}$ and let $q\in[1,\infty]$. Then, $\exists\:C,h_{0} > 0$ such that $\forall n\geq N$,  $\forall X_{n}\subseteq\X$ with $h_{X_{n}}\leq h_{0}$ and $\forall s\in[0,(\tau_{f}\wedge\tau^{-}_{k})^{*}]$ the approximation error is bounded as
	\begin{align*}
		&\mathbb{E}\left[\bignorm{f -R_{f,\sigma^2_n,\varepsilon}^{m}(\theta_{n})}_{W^{s}_{q}(\X)}\right] \\
		& \qquad \leq \; C\Lambda_{s,q}\bigg[h_{X_{n}}^{\left(\tau_{f}\wedge\tau^{-}_{k}\right)-s-d\left(\frac{1}{2}-\frac{1}{q}\right)_{+}}\rho_{X_{n}}^{(\tau^{+}_{k} - \tau_{f})_{+}}\left(\norm{f}_{W^{\tau_{f}}_{2}(\X)} + \bignorm{m(\theta_{n})}_{W^{\tau_{f}}_{2}(\X)}\right) \\
		& \qquad  \qquad \qquad \qquad+ h_{X_{n}}^{\tau^{-}_{k} -s-d \left(\frac{1}{2}-\frac{1}{q}\right)_{+}}\sigma_n^{-1}\bbE[\norm{\varepsilon}_{2}] \\
	& \qquad \qquad  \qquad \qquad + h_{X_{n}}^{\frac{d}{\gamma}-s}\bbE\left[\bignorm{(f-R_{f,\sigma^2_n,\varepsilon}(\theta_{n}))_{X_{n}}}_{2}\right]\\
	&\qquad \qquad  \qquad \qquad + h_{X_{n}}^{\frac{d}{\gamma}-s}\bbE\left[\bignorm{(m(\theta_{n})-R_{m(\theta_{n}),\sigma^2_n,\varepsilon}(\theta_{n}))_{X_{n}}}_{2}\right]\bigg],
	\end{align*}
	where $C = C\left(\X,d,q,\tau_{f},\Theta_{N}^{*}\right)$, $h_{0} = h_{0}\left(R,\delta,d,\tau_{f},\Theta_{N}^{*}\right)$, $\gamma = 2\vee q$.
\end{theorem}

\begin{proof}
Expectation with respect to $\varepsilon$ shall be taken at the final step. By the definition of $R^{m}_{f,\sigma^2_n,\varepsilon}(\theta_{n})$ 
\begin{align*}
	R^{m}_{f,\sigma^2_n,\varepsilon}(\theta_{n}) = m(\theta_{n}) + R_{f,\sigma^2_n,\varepsilon}(\theta_{n}) - R_{m,\sigma^2_n,\varepsilon}(\theta_{n}) + R_{0,\sigma^2_n,\varepsilon}(\theta_{n}),
\end{align*}
therefore
\begin{align}
	\bignorm{f-R_{f,\sigma^2_n, \epsilon}^{m}(\theta_{n})}_{W^{s}_{q}(\X)} 
	&\leq \bignorm{f-R_{f,\sigma^2_n, \epsilon}(\theta_{n})}_{W^{s}_{q}(\X)}\nonumber \\ 
	& + \bignorm{m(\theta_{n})-R_{m(\theta_{n}),\sigma^2_n, \epsilon}(\theta_{n})}_{W^{s}_{q}(\X)} \nonumber \\ 
	& + \bignorm{R_{0,\sigma^2_n,\varepsilon}}_{W^{s}_{q}(\X)}, \label{eq:start_deterministic_thm}
\end{align}
so as in the proof of Theorem \ref{thm:Misspecified_Sobolev_Interpolation}, see \eqref{eq:start_interpolation_thm}, without loss of generality it suffices to only consider the case $m = 0$. Use Theorem \ref{thm:Arcangeli_SI} on $f-R_{f,\sigma^2_n,\varepsilon}(\theta_{n})$ to see $\exists h_{1} > 0$ such that for $h_{X_{n}} \leq h_{1}$ and any $s \in [0,\left(\tau_{f}\wedge\tau_{k}^{-}\right)^{*}]$
\begin{align}
	\bignorm{f-R_{f,\sigma^2_n,\varepsilon}(\theta_{n})}_{W^{s}_{q}(\X)} & \leq C_{1}\Lambda_{s,q}\bigg(h_{X_{n}}^{(\tau_{f}\wedge\tau(\theta_{n}))-s-d\left(\frac{1}{2}-\frac{1}{q}\right)_{+}} \bignorm{f-R_{f,\sigma^2_n,\varepsilon}(\theta_{n})}_{W^{\tau_{f}\wedge\tau(\theta_{n})}_{2}(\X)} \nonumber\\
	& \qquad \qquad + h_{X_{n}}^{\frac{d}{\gamma} - s}\bignorm{\left(f-R_{f,\sigma^2_n,\varepsilon}(\theta_{n})\right)_{X_n}}_{2}\bigg), \label{eq:proof_noisy_1stbound}
\end{align}
where $C_{1} = C_{1}(\X,d,q,\tau_{f},\Theta_{N}^{*})$, $h_{1} = h_{1}(R,\delta,d,\tau_{f},\Theta_{N}^{*})$ and $\gamma = \max(2,q)$. The rest of the proof is spent bounding the $W^{\tau_{f}\wedge\tau(\theta_{n})}_{2}(\X)$ norm term.

For the case $\tau(\theta_{n})\leq \tau_{f}$, the triangle inequality and Lemma \ref{lem:Regression_Best_Approximation} can be employed
\begin{align*}
	\bignorm{f - R_{f,\sigma^2_n,\varepsilon}(\theta_{n})}_{W^{\tau_{f}\wedge\tau(\theta_{n})}_{2}(\X)} & = \bignorm{f - R_{f,\sigma^2_n,\varepsilon}(\theta_{n})}_{W^{\tau(\theta_{n})}_{2}(\X)} \\
	& \leq \norm{f}_{W^{\tau(\theta_{n})}_{2}(\X)} + \bignorm{R_{f,\sigma^2_n,\varepsilon}(\theta_{n})}_{W^{\tau(\theta_{n})}_{2}(\X)} \\
	& \leq C_{2}\big(\norm{f}_{W^{\tau_{f}}_{2}(\X)} + \sigma_{n}^{-1}\norm{\varepsilon}_{2}\big),
\end{align*}
with $C_{2}$ bounding the ratio of norm equivalence constants which facilitates the use of RKHS norms in Lemma \ref{lem:Regression_Best_Approximation}, this is analogous to the use of ratio of norm equivalence constants in \eqref{eq:larger_tau_result_4}. Then, combined with \eqref{eq:proof_noisy_1stbound}
\begin{align}
\bignorm{f -R_{f,\sigma^2_n,\varepsilon}(\theta_{n})}_{W^{s}_{q}(\X)} \leq \; C'\Lambda_{s,q}\bigg[&h_{X_{n}}^{\tau(\theta_{n})-s-d\left(\frac{1}{2}-\frac{1}{q}\right)_{+}}\norm{f}_{W^{\tau_{f}}_{2}(\X)} \nonumber\\
		& + h_{X_{n}}^{\tau(\theta_{n}) -s-d \left(\frac{1}{2}-\frac{1}{q}\right)_{+}}\sigma_n^{-1}\norm{\varepsilon}_{2} \label{eq:smooth_version} \\
	& + h_{X_{n}}^{\frac{d}{\gamma}-s}\bignorm{(f-R_{f,\sigma^2_n,\varepsilon}(\theta_{n}))_{X_{n}}}_{2}\bigg],\nonumber
\end{align}
which recovers the desired result for this case. 

For the case when $\tau(\theta_{n}) > \tau_{f}$ first apply the triangle inequality 
\begin{align}
    \bignorm{f-R_{f,\sigma_{n}^{2},\varepsilon}(\theta_{n})}_{W^{s}_{q}(\X)}\leq \bignorm{f-R_{f}(\theta_{n})}_{W^{s}_{q}(\X)} + \bignorm{R_{f}(\theta_{n})-R_{f,\sigma_{n}^{2},\varepsilon}(\theta_{n})}_{W^{s}_{q}(\X)},\label{eq:tri_thm4}
\end{align}
where $R_{f}(\theta_{n})$ is an interpolant of $f$ at the data points. The first term on the right hand side of \eqref{eq:tri_thm4} can be bounded by a direct application of Theorem \ref{thm:Misspecified_Sobolev_Interpolation} (the bound for GPs in the interpolation setting with well-specified likelihood). The second term can be bounded using \eqref{eq:smooth_version} by replacing $f$ with $R_{f}(\theta_{n})$. This can be done since $f$ and $R_{f}(\theta_{n})$ agree at the data points and $R_{f}(\theta_{n})$ and $R_{f,\sigma_{n}^{2},\varepsilon}(\theta_{n})$ have the same smoothness. Combining these two bounds yields
\begin{align}
    & \bignorm{f -R_{f,\sigma^2_n,\varepsilon}(\theta_{n})}_{W^{s}_{q}(\X)} \nonumber\\
    & \leq C'\Lambda_{s,q}\bigg[h_{X_{n}}^{\tau(\theta_{n})-s-d\left(\frac{1}{2}-\frac{1}{q}\right)_{+}}\bignorm{R_{f}}_{W^{\tau(\theta_{n})}_{2}(\X)} + h_{X_{n}}^{\tau(\theta_{n}) -s-d \left(\frac{1}{2}-\frac{1}{q}\right)_{+}}\sigma_n^{-1}\norm{\varepsilon}_{2} \nonumber\\
	& + h_{X_{n}}^{\frac{d}{\gamma}-s}\bignorm{\left(R_{f}(\theta_{n})-R_{f,\sigma^2_n,\varepsilon}(\theta_{n})\right)_{X_{n}}}_{2}\bigg],\nonumber\\
	& \leq C''\Lambda_{s,q}\bigg[h_{X_{n}}^{\tau_{f}-s-d\left(\frac{1}{2}-\frac{1}{q}\right)_{+}}\rho_{X_{n}}^{\left(\tau(\theta_{n})-\tau_{f}\right)}\bignorm{f}_{W^{\tau_{f}}_{2}(\X)}  + h_{X_{n}}^{\tau(\theta_{n}) -s-d \left(\frac{1}{2}-\frac{1}{q}\right)_{+}}\sigma_n^{-1}\norm{\varepsilon}_{2} \nonumber \\
	& + h_{X_{n}}^{\frac{d}{\gamma}-s}\bignorm{\left(f-R_{f,\sigma^2_n,\varepsilon}(\theta_{n})\right)_{X_{n}}}_{2}\bigg],\label{eq:band_lim}
\end{align}
where \eqref{eq:band_lim} uses the fact that $R_{f}(\theta_{n})$ interpolates $f$ to obtain the final term and employs the proof technique involving band limited functions by \cite[Theorem 4.2]{Narcowich2006} to express the $W^{\tau(\theta_{n})}_{2}(\X)$ norm in terms of the $W^{\tau_{f}}_{2}(\X)$ norm.

Combining \eqref{eq:band_lim} with the bound obtained when applying Theorem \ref{thm:Misspecified_Sobolev_Interpolation} to the first term on the right hand side of \eqref{eq:tri_thm4} then taking appropriate upper and lower bounds of $\tau(\theta_{n})$ completes the proof.
\end{proof}

Combining the following lemma with Theorem \ref{thm:Misspecified_Sobolev_corruption} completes the proof of Theorem \ref{thm:Misspecified_Sobolev_Det_Regression}. 

\begin{lemma}\label{lemma:residuals}
	Suppose the assumptions of Theorem \ref{thm:Misspecified_Sobolev_Det_Regression} hold then for any $f\in W^{\tau_{f}}_{2}(\X)$
	\begin{align*}
		\bignorm{\left(f-R_{f,\sigma^2_n,\varepsilon}(\theta_{n})\right)_{X_{n}}}_{2}\leq C\left(\norm{\varepsilon}_{2} + \sigma_n q_{X_{n}}^{-(\tau^{+}_{k}-\tau_{f})_{+}}\norm{f}_{W^{\tau_{f}}_{2}(\X)}\right),
	\end{align*}
	where $C = C\left(\X,d,\tau_{f},\Theta_{N}^{*}\right)$.	
\end{lemma}

\begin{proof}
    As discussed in the proof by \citet[Theorem 4.2]{Narcowich2006} for each $n$ there exists a band-limited function $f_{\alpha_{n}}$, where $\alpha_{n}$ is the bandwidth and depends on $q_{X_{n}}$, such that $f_{\alpha}$ equals $f$ at the points $X_{n}$ and $\norm{f_{\alpha}}_{W^{\tau(\theta_{n})}_{2}(\X)}\leq C_{1}q_{X_{n}}^{-(\tau(\theta_{n})-\tau_{f})_{+}}$ for some $C_{1}=C_{1}(\X,d,\tau_{f},\tau(\theta_{n}))$. Using this,
    \begin{align}
        \bignorm{\left(f-R_{f,\sigma^2_n,\varepsilon}(\theta_{n})\right)_{X_{n}}}_{2}  &= \bignorm{\left(f_{\alpha}-R_{f_{\alpha},\sigma^2_n,\varepsilon}(\theta_{n})\right)_{X_{n}}}_{2}\\
        & \leq \norm{\varepsilon}_{2} + \left(\norm{\varepsilon}^{2}_{2} + \sigma_{n}^2 \norm{f_{\alpha_{n}}}_{\mathcal{H}_{k(\theta_{n})}(\X)}^{2}\right)^{\frac{1}{2}}\label{eq:lem18}\\
        & \leq C\left(\norm{\varepsilon}_{2} + \sigma_n q_{X_{n}}^{-(\tau^{+}_{k}-\tau_{f})_{+}}\norm{f}_{W^{\tau_{f}}_{2}(\X)}\right)\label{eq:norm_eq}
    \end{align}
    where \eqref{eq:lem18} used Lemma \ref{lem:Regression_Best_Approximation} and \eqref{eq:norm_eq} used the norm equivalence of the RKHS to a Sobolev space and the aforementioned property of $f_{\alpha_{n}}$ to obtain the $q_{X_{n}}$ term.
\end{proof}

\section{Proof of Theorem \ref{thm:misspec_likelihood_inter}}

We will denote by $R_{\varepsilon}$ the kernel interpolant of the the noise, meaning $R_{\varepsilon}(x) = k_{xX}k_{XX}^{-1}\varepsilon$. 

\begin{lemma}\label{lem:random_rkhs_norm}
	Let $k$ be a $\tau$-smooth kernel for $\tau > d/2$, $\varepsilon\in\bbR^{n}$ and $X_{n}\subset\X$. Then, $\exists C>0$ such that $\norm{R_{\varepsilon}}_{\calH_{k}(\X)}\leq C\norm{\varepsilon}_{2}q_{X_{n}}^{-(\tau-d/2)}$ for some $C = C(d,k)$ with the dependence on $k$ entering through the RKHS norm equivalence constants. 
\end{lemma}
\begin{proof}
	Denote by $\lambda_{\text{min}}(A),\lambda_{\text{max}}(A)$ the minimum and maximum eigenvalues of some matrix $A$. Then:
	\begin{align}
		\norm{R_{\varepsilon}(\theta_{n})}_{H_{k}(\X)}^{2} = \varepsilon^{\top}k_{XX}^{-1}\varepsilon \leq\norm{\varepsilon}_{2}^{2} \sup_{x\in\bbR^{n}, x\neq 0}& \frac{x^{\top}k_{XX}^{-1}x}{\norm{x}_{2}^{2}}  = \norm{\varepsilon}_{2}^{2}\lambda_{\text{max}}(k_{XX}^{-1}) \label{eq:courant}\\
		& = \norm{\varepsilon}_{2}^{2}\lambda_\text{min}(k_{XX})^{-1}  \leq C\norm{\varepsilon}_{2}^{2}q_{X_{n}}^{-(2\tau-d)},\label{eq:min_eig_bound}
	\end{align}
	where the first inequality in \eqref{eq:courant} is by the reproducing property and the last inequality is by the Rayleigh-Ritz theorem \citep{Horn1985}, \eqref{eq:min_eig_bound} is by using the bounds on minimum eigenvalues of kernel matrices discussed by \citet[Theorem 12.3]{Wendland2005} which are applicable since $k$ is $\tau$-smooth. See e.g.\ \citep[Section 3]{Narcowich2006} for further discussion. 
\end{proof}

To prove Theorem \ref{thm:misspec_likelihood_inter} proceed as in the proof of Theorem \ref{thm:Misspecified_Sobolev_corruption} up to \eqref{eq:proof_noisy_1stbound} and note the residual term is simply $\norm{\varepsilon}_{2}$ since $R_{f,0,\varepsilon}(\theta_{n})$ interpolates the corrupted data, rather than $f_{X}$. So $\exists \:C_{1}, h_{1} > 0$ with $C_{1} = C_{1}(\X,d,q,\tau_{f},\Theta_{N}^{*})$, $h_{1} = h_{1}(R,\delta,d,\tau_{f},\Theta_{N}^{*})$ such that for any $X_{n}$ with $h_{X_{n}}\leq h_{1}$
\begin{align}
&\bignorm{f-R_{f,0,\varepsilon}(\theta_{n})}_{W^{s}_{q}(\X)}  \nonumber\\
& \qquad\leq C_{1}\Lambda_{s,q}\bigg(h_{X_{n}}^{(\tau_{f}\wedge\tau(\theta_{n}))-s-d\left(\frac{1}{2}-\frac{1}{q}\right)_{+}} \bignorm{f-R_{f,0,\varepsilon}(\theta_{n})}_{W^{\tau_{f}\wedge\tau(\theta_{n})}_{2}(\X)}  + h_{X_{n}}^{\frac{d}{\gamma} - s}\norm{\varepsilon}_{2}\bigg). \label{eq:1st_bound}
\end{align}

First consider the case when $\tau_{f}\geq\tau(\theta_{n})$ 
\begin{align}
	\norm{f-R_{f,0,\varepsilon}(\theta_{n})}_{W^{\tau_{f}\wedge\tau(\theta_{n})}_{2}(\X)} & \leq \norm{f-R_{f}(\theta_{n})}_{W^{\tau_{f}}_{2}(\X)} + \norm{R_{\varepsilon}(\theta_{n})}_{W^{\tau(\theta_{n})}_{2}(\X)}\label{eq:tri}\\
	& \leq \norm{f}_{W^{\tau_{f}}_{2}(\X)} + C_{2}\norm{\varepsilon}_{2}q_{X_{n}}^{-(\tau(\theta_{n})-\frac{d}{2})},\label{eq:py}
\end{align}
where \eqref{eq:tri} is the triangle inequality and \eqref{eq:py} is by Lemma \ref{lem:pythag} and Lemma \ref{lem:random_rkhs_norm} with $C_{2} = C(d,\Theta_{N}^{*})$, the dependency on $\Theta_{N}^{*}$ manifested by using the ratio of norm equivalence constants to move from Sobolev to RKHS norm. Combining this with \eqref{eq:1st_bound}
\begin{align}
	&\bignorm{f-R_{f,0,\varepsilon}(\theta_{n})}_{W^{s}_{q}(\X)}  \nonumber\nonumber\\
& \quad\leq C_{1}\Lambda_{s,q}\bigg(h_{X_{n}}^{\tau(\theta_{n})-s-d\left(\frac{1}{2}-\frac{1}{q}\right)_{+}} \big(\norm{f}_{W^{\tau_{f}}_{2}(\X)} + C_{2}\norm{\varepsilon}_{2}q_{X_{n}}^{-(\tau(\theta_{n})-\frac{d}{2})}\big)  + h_{X_{n}}^{\frac{d}{\gamma} - s}\norm{\varepsilon}_{2}\bigg) \nonumber \\
& \quad\leq C_{3}\Lambda_{s,q}\bigg(h_{X_{n}}^{\tau(\theta_{n})-s-d\left(\frac{1}{2}-\frac{1}{q}\right)_{+}} \norm{f}_{W^{\tau_{f}}_{2}(\X)} + \big(h_{X_{n}}^{\frac{d}{\gamma}-s} + \rho_{X_{n}}^{\tau(\theta_{n})-\frac{d}{2}}h_{X_{n}}^{\frac{d}{2} - s - d\left(\frac{1}{2}-\frac{1}{q}\right)_{+}}\big)\norm{\varepsilon}_{2}\bigg)\label{eq:abs}\\
& \quad \leq C_{4}\Lambda_{s,q}\bigg(h_{X_{n}}^{\tau(\theta_{n})-s-d\left(\frac{1}{2}-\frac{1}{q}\right)_{+}} \norm{f}_{W^{\tau_{f}}_{2}(\X)} + h_{X_{n}}^{\frac{d}{\gamma}-s}\rho_{X_{n}}^{\tau(\theta_{n})-\frac{d}{2}}\norm{\varepsilon}_{2}\bigg),\label{eq:h_exp}
\end{align}
where \eqref{eq:abs} is absorbing the constants to the front and \eqref{eq:h_exp} is because the exponents of the $h_{X_{n}}$ terms that are multiplied by $\norm{\varepsilon}_{2}$ are the same and $\rho_{X_{n}}\geq 1$. Taking upper and lower bounds of $\tau(\theta_{n})$ completes the proof for this case. 

Now consider $\tau_{f} < \tau(\theta_{n})$. In a similar fashion to the second part of the proof of Theorem \ref{thm:Misspecified_Sobolev_corruption}, we use the triangle inequality
\begin{align}
    \bignorm{f-R_{f,0,\varepsilon}(\theta_{n})}_{W^{s}_{q}(\X)}\leq \bignorm{f-R_{f}(\theta_{n})}_{W^{s}_{q}(\X)} + \bignorm{R_{f}(\theta_{n})-R_{f,0,\varepsilon}(\theta_{n})}_{W^{s}_{q}(\X)},\label{eq:ti}
\end{align}
where the first term on the right hand side can be bounded by Theorem \ref{thm:Misspecified_Sobolev_Interpolation} since it does not involve observation corruption, and the second term on the right hand side can be bounded by \eqref{eq:h_exp} by replacing $f$ with $R_{f}(\theta_{n})$ since $R_{f}(\theta_{n})$ and $R_{f,0,\varepsilon}(\theta_{n})$ have the same smoothness. Therefore
\begin{align}
    &\bignorm{f-R_{f,0,\varepsilon}(\theta_{n})}_{W^{s}_{q}(\X)}\nonumber\\
    & \qquad\leq C_{5}\Lambda_{s,q}\bigg(h_{X_{n}}^{\tau(\theta_{n})-s-d\left(\frac{1}{2}-\frac{1}{q}\right)_{+}} \norm{R_{f}(\theta_{n})}_{W^{\tau(\theta_{n})}_{2}(\X)} + h_{X_{n}}^{\frac{d}{\gamma}-s}\rho_{X_{n}}^{\tau(\theta_{n})-\frac{d}{2}}\norm{\varepsilon}_{2}\bigg)\label{eq:thm1_h_exp}\\
    & \qquad\leq C_{6}\Lambda_{s,q}\bigg(h_{X_{n}}^{\tau(\theta_{n})-s-d\left(\frac{1}{2}-\frac{1}{q}\right)_{+}} q_{X_{n}}^{-(\tau(\theta_{n})-\tau_{f})}\norm{f}_{W^{\tau_{f}}_{2}(\X)} + h_{X_{n}}^{\frac{d}{\gamma}-s}\rho_{X_{n}}^{\tau(\theta_{n})-\frac{d}{2}}\norm{\varepsilon}_{2}\bigg)\label{eq:bern}\\
    & \qquad\leq C_{6}\Lambda_{s,q}h_{X_{n}}^{\frac{d}{\gamma}-s}\rho_{X_{n}}^{\tau(\theta_{n})-\tau_{f}}\bigg(h_{X_{n}}^{\tau_{f}-\frac{d}{2}} \norm{f}_{W^{\tau_{f}}_{2}(\X)} + \rho_{X_{n}}^{\tau_{f}-\frac{d}{2}}\norm{\varepsilon}_{2}\bigg),\label{eq:collect}
\end{align}
where \eqref{eq:thm1_h_exp} is applying Theorem \ref{thm:Misspecified_Sobolev_Interpolation} and \eqref{eq:h_exp} to the terms on the right hand side of \eqref{eq:ti}. Then \eqref{eq:bern} uses, as was done in the proof of Theorem \ref{thm:Misspecified_Sobolev_corruption}, the proof technique involving band-limited functions by \citet[Theorem 4.2]{Narcowich2006}. Finally \eqref{eq:collect} is collecting the mesh ratio and fill distance terms to the front. Taking appropriate upper and lower bounds for $\tau(\theta_{n})$ in terms of $\tau_{k}^{+},\tau_{k}^{-}$ completes the proof of Theorem \ref{thm:misspec_likelihood_inter}.

\section{Proof of Theorem \ref{thm:Misspecified_Sobolev_Random_Regression}} \label{appendix:gaussian_noise}

The proof of Theorem \ref{thm:Misspecified_Sobolev_Random_Regression} is simply combination of Theorem \ref{thm:Misspecified_Sobolev_corruption} with a bound on the residual terms and substituting in a bound for $\bbE[\norm{\varepsilon}_{2}]$. The bound on the residual terms is obtained from an adaptation of the result by \citet[Theorem 1, Theorem 5]{VanderVaart2011} to the case of altering hyperparamters. Proving this adaptation is a tedious matter of checking that the constants involved in the bound by \citet[Theorem 1, Theorem 5]{VanderVaart2011}, which are different for each parameter value, may be controlled given our assumptions on the hyperparameters. The adaptation is stated next along with an explanation of how it is used to prove Theorem \ref{thm:Misspecified_Sobolev_Random_Regression} and then a proof of the adaptation is given. 

\begin{proposition}\label{prop:residual_contraction}
	Suppose Assumptions 1-\ref{ass:target_extension} then $\exists C = C\left(\norm{f}_{W^{\tau_{f}}_{2}(\X)},\Theta_{N}^{*}\right)$ such that for $n\geq N$ 
	\begin{align*}
		\mathbb{E}\left[\bignorm{\left(f - R_{f,\sigma^{2},\varepsilon}(\theta_n)\right)_{X_{n}}}_{2}\right] 
		& \leq C\left(n^{\left(\frac{1}{2}-\frac{\tau_{f}}{2\tau_{k}^{+}}\right)_{+}\vee \frac{d}{4\tau_{k}^{-}}}\right).
	\end{align*}
\end{proposition}

Direct substitution of this bound into the residuals in Theorem \ref{thm:Misspecified_Sobolev_corruption} and noting that $\sigma^{-1}\bbE[\norm{\varepsilon}_{2}] \leq n^{\frac{1}{2}}$ completes the proof of Theorem \ref{thm:Misspecified_Sobolev_Random_Regression}. Proposition \ref{prop:residual_contraction} is proved by using Jensen's inequality on Theorem \ref{thm:Sobolev_Contraction} in combination with Corollary \ref{cor:gen_inv_bound} to obtain the desired bound. The rest of this section shall prove these intermediate results. 

Before starting the details of the proof of Proposition \ref{prop:residual_contraction} we recall the definition of \emph{H\"older spaces of functions} $C^{\tau}(\X)$. For $\tau > 0$ and $\X\subseteq\bbR^{d}$ an open set, $C^{\tau}(\X)$ is the space of functions $f\colon\X\rightarrow\bbR$ with $\norm{f}_{C^{\tau}(\X)}<\infty$ where
	\begin{align*}
		\norm{f}_{C^{\tau}(\X)} = \max_{m\colon \lvert m\rvert \leq \lfloor \tau - 1\rfloor}\sup_{x\in\X}\left| D^{m}f(x)\right| + \max_{m\colon \lvert m\rvert \leq \lfloor \tau - 1\rfloor}\sup_{\substack{x,y\in\X \\ x\neq y}}\frac{\lvert D^{m}f(x)-D^{m}f(y)\rvert}{\norm{x-y}_{2}^{\tau - \lfloor\tau - 1\rfloor}},
	\end{align*}
where $m = (m_{1},\ldots,m_{d})$ is a multi-index, $\lvert m\rvert = \sum_{i=1}^{d}m_{i}$ and $D^{m}$ is the partial differential operator corresponding to $m$. Now the framework by \citet{VanderVaart2011} is presented which views the Gaussian process as a measure on function space. The techniques discussed are detached from the results in the present paper and are discussed only to prove Proposition \ref{prop:residual_contraction}, for further details of their origin and use in Bayesian nonparametrics see e.g.\ \citep{Ghosal2017}.

Let $\Pi_{k(\theta_{n})}$ denote the probability measure associated with a GP with zero mean and kernel $k(\theta_{n})$ over $\X$. Set we $\Pi_{\theta_{n}} = \Pi_{k(\theta_{n})}$ for ease of notation. Given a target function $f$ and a set of points $X_{n} = \{x_{i}\}_{i=1}^{n}$ and observations $y_{i} = f(x_{i}) + \varepsilon_{i}$ with $\varepsilon_{i}$ i.i.d $\mathcal{N}(0,\sigma^{2})$ denote the posterior distribution of $\Pi_{\theta_{n}}$ given $\{y_{i}\}_{i=1}^{n}$ as $\Pi_{\theta_{n}}(\cdot | y_{1:n})$. For $\varepsilon > 0$ and $f$ a continuous function over the closure of $\X$ define the concentration function
\begin{align*}
	\phi_{\theta_{n},f}(\varepsilon) = \inf_{\substack{h\in\calH_{k(\theta_{n})}(\X) \\ \norm{h - f}_{L^{\infty}(\X)} < \varepsilon}}\frac{1}{2}\norm{h}_{\calH_{k(\theta_{n})}(\X)}^{2} - \log\Pi_{\theta_{n}}(g\: :\: \norm{g}_{L^{\infty}(\X)} < \varepsilon).
\end{align*}
The first term is called the decentering function and the second the small ball probability. This is finite if and only if $f$ is contained in the closure of $\calH_{k(\theta_{n})}(\X)$ with respect to the supremum norm, which will be true under the assumptions of the theorems in Section \ref{sec:noisy_convergence}. The next result by \citet[Theorem 1]{VanderVaart2011} shows the residuals may be controlled by the concentration function.

\begin{theorem}\label{thm:Sobolev_Contraction}
	Let $\X$ be a compact set then $ \exists C > 0$ such that for every $f\in C(\X)$ 
	\begin{align*}
		\frac{1}{n}\mathbb{E}\left[\int\bignorm{(g-f)_{X_{n}}}_{2}^{2}d\Pi_{\theta_{n}}(g|y_{1:n})\right]\leq C\psi_{\theta_{n},f}^{-1}(n)^{2},
	\end{align*}
	where the expectation is being taken with respect to the noise, $\psi_{\theta_{n},f}(\varepsilon) = \phi_{\theta_{n},f}(\varepsilon)/\varepsilon^{2}$ and $\psi^{-1}_{\theta_{n},f}$ is the generalised inverse of $\psi_{\theta_{n},f}$.
\end{theorem}

Compactness of $\X$ is assumed whereas in Theorem \ref{thm:Misspecified_Sobolev_Random_Regression} we assumed $\X$ is open. This is not an issue since it is assumed the target function can be extended to all of $\bbR^{d}$ so Theorem \ref{thm:Sobolev_Contraction} may be applied to the restriction of the extension to the closure of $\X$, which is compact and contains all the observation points. The decentering function and the small ball probability needs to be bounded. The decentering function is bounded in Lemma \ref{lem:decentering} and the small ball probability in Lemma \ref{lem:small_ball} which requires more technical work. 

Specifically the decentering term is dealt with by upper bounding norm equivalence constants that occur in the proof by \citet[Lemma 4]{VanderVaart2011} when performing the kernel convolution approximation argument in that proof, this is summarised in the next lemma.

\begin{lemma}\label{lem:decentering}
	Let $f$ be the restriction to the closure of $\X$ of some $f^{\circ}\in C^{\tau_{f}}(\bbR^{d})\cap W^{\tau_{f}}_{2}(\bbR^{d})$ with $\tau(\theta_{n}) > \tau_{f} > d/2$. Then, $\exists C = C\left(\norm{f}_{W^{\tau_{f}}_{2}(\X)},\Theta_{N}^{*}\right)$ such that for $\varepsilon < 1$
	\begin{align*}
		\inf_{\substack{h\in\calH_{k(\theta_{n})}(\X) \\ \norm{h - f}_{L^{\infty}(\X)} < \varepsilon}}\frac{1}{2}\norm{h}_{\calH_{k(\theta_{n})}(\X)}^{2} 
		\;\leq \; C\varepsilon^{-\frac{2(\tau(\theta_{n})-\tau_{f})}{\tau_{f}}} 
		\; \leq \; C\varepsilon^{-2\frac{(\tau_{k}^{+}-\tau_{f})}{\tau_{f}}}.
	\end{align*}
\end{lemma}

If $\tau_{f}\geq\tau(\theta_{n})$ then $f\in\calH_{k(\theta_{n})}(\X)$ therefore we could take $h=f$ and the bound would be $\frac{1}{2}\norm{f}_{W^{\tau(\theta_{n})}_{2}(\X)}$ which has no dependence on $\varepsilon$ so in this case the growth of the concentration function is dictated entirely by the small ball probability term. 

The small ball probability bound requires the result by \citet[Theorem 1.2]{Linde1999} which relates small ball probabilities to the metric entropy of the unit ball of the RKHS corresponding to the kernel. Metric entropy is a method of measuring the size of a given function space denoted $H(M,\varepsilon)$ and is defined as the logarithm of the $\varepsilon$-covering number of $M$, for more discussion see e.g.\ \citep[Chapter 2.3]{Gine2016}. A bound on metric entropy is given by \citet[Theorem 4.3.36]{Gine2016} which illuminates the way the hyperparameters effect the bound. The constants in the proof can easily be bounded by replacing $\tau(\theta_{n})$ by $\tau^{+}_{k}$ and $\tau^{-}_{k}$ where appropriate and doing so yields the next lemma.

\begin{lemma}\label{lem:metric_entropy}
	Fix $N\in\bbN$ and suppose Assumptions 1-\ref{ass:target_extension} hold. Let $\calH_{k(\theta_{n})}^{(1)}(\overline{\X})$ denote the unit ball of the RKHS of $k(\theta_{n})$ over the closure of $\X$. Then $\exists C_{met} = C_{met}\left(\Theta_{N}^{*}\right)$ such that $\forall n\geq N$ and $\forall \varepsilon < 1$ 
	\begin{align*}
	H\left(\calH_{k(\theta_{n})}^{(1)}\left(\overline{\X}\right),\varepsilon\right) \leq C_{met}\;\varepsilon^{-\frac{d}{\tau(\theta_{n})}}.
\end{align*}
\end{lemma}

The proof by \citet[Theorem 5]{VanderVaart2011} is now followed to link metric entropy to small ball probability. This will involve going through auxiliary results by \citet{Linde1999} to make sure the possible altering hyperparameters result in constants that are controlled, this is a tedious process and the referenced paper should be consulted for greater context. Before this is started note that the two auxiliary results by \citet[Lemma 2.1, Lemma 2.2]{Linde1999}  used to link entropy numbers to GPs do not depend on the hyperparameter choices since they hold with constants not depending on the smoothness of the RKHS, see e.g.\ \citep[Theorem 9.1]{Pisier1989} and \citep[Page 1315]{Artstein2004}. The first step is to use \citep[Proposition 2.4]{Linde1999} in combination with the bound we have derived for metric entropy to get that $\forall \gamma > 2d/(2\tau(\theta_{n})-d)$, $ \exists\: C(\theta_{n},\gamma) > 0$ such that
\begin{align}
	\phi_{\theta_{n}}(\varepsilon) \coloneqq -\log\Pi_{\theta_{n}}\left(\norm{f}_{\infty} \leq \varepsilon\right) \leq C(\theta_{n},\gamma)\varepsilon^{-\gamma}.\label{eq:metric_bound}
\end{align}
Next we explain the result by \citet[Proposition 3.1]{Linde1999}.  First, using the result by \citet[Equation 3.4]{Linde1999} and Lemma \ref{lem:metric_entropy}
\begin{align}
	\phi_{\theta_{n}}(\varepsilon) & \leq \log 2+ H\left(\calH_{k(\theta_{n})}^{(1)},\varepsilon (8\phi_{\theta_{n}}(\varepsilon/2))^{-\frac{1}{2}}\right) \nonumber\\
	& \leq \log 2 + C_{met}\varepsilon^{-\frac{d}{\tau(\theta_{n})}}8^{\frac{d}{2\tau(\theta)}}\phi_{\theta_{n}}\left(\varepsilon/2\right)^{\frac{d}{2\tau(\theta_{n})}},\label{eq:small_ball_bound}
\end{align}
then once $\varepsilon$ gets smaller than some $\varepsilon^{*}$ the second term on the right hand side of \eqref{eq:small_ball_bound} becomes greater than $L\log 2$ for some constant $L > 0$ \citep[Equation 3.4]{Linde1999}. This argument does not consider changing hyperparameters. Indeed, for a fixed constant $L > 0$ for different hyperparameters $\theta_{n}$ we might need different $\varepsilon^{*}_{n}$ to conclude that the second term is greater than $L\log 2$. Assumption 5 introduces the required uniformity by allowing us to say that once $\varepsilon$ is small enough \eqref{eq:small_ball_bound} can be bounded by a constant times the second term in \eqref{eq:small_ball_bound} for all hyperparameter choices. Specifically, by Assumption \ref{ass:small_ball} we know that if $\varepsilon < c$ then $\phi_{\theta_{n}}(\varepsilon/2)\geq\alpha_{N}$, therefore if we set
\begin{align}
	\varepsilon^{*} \coloneqq \min\left(c,\left(\alpha_{N}^{\frac{d}{2\tau_{k}^{+}}}(\log 2)^{-1}\right)^{\frac{\tau_{k}^{-}}{d}}\right),\label{eq:ep_star}
\end{align}
then for $\varepsilon < \varepsilon^{*}$, we have 
\begin{align*}
	\phi_{\theta_{n}}(\varepsilon) 
	& \leq \log 2 + C_{met}\; \varepsilon^{-\frac{d}{\tau(\theta_{n})}}8^{\frac{d}{2\tau(\theta_{n})}}\phi_{\theta_{n}}\left(\varepsilon/2\right)^{\frac{d}{2\tau(\theta_{n})}} \\
	& \leq (C_{met} + 1) \; \varepsilon^{-\frac{d}{\tau(\theta_{n})}}8^{\frac{d}{2\tau(\theta_{n})}}\phi_{\theta_{n}}\left(\varepsilon/2\right)^{\frac{d}{2\tau(\theta_{n})}}. 
\end{align*}

Now take logarithms and employ the iterative procedure from the proof by \citet[Proposition 3.1]{Linde1999}. Taking logarithms gives
\begin{align*}
	\log\phi_{\theta_{n}}(\varepsilon) \leq \frac{d}{2\tau(\theta_{n})}\log\phi_{\theta}\left(\varepsilon/2\right) + \log\chi_{n}(\varepsilon),
\end{align*}
where $\chi_{n}(\varepsilon) = (C_{met}+1)\varepsilon^{-\frac{d}{\tau(\theta_{n})}}8^{\frac{d}{2\tau(\theta_{n})}}$. Now iterate this inequality so that for any $m\in\bbN$
\begin{align}
	\log\phi_{\theta_{n}}(\varepsilon)\leq \left(\frac{d}{2\tau(\theta_{n})}\right)^{m}\log\phi_{\theta_{n}}\left(2^{-m}\varepsilon \right) + \sum_{j=0}^{m-1}\left(\frac{d}{2\tau(\theta_{n})}\right)^{j}\log\chi\left(  2^{-j} \varepsilon \right),\label{eq:iterated_inequality}
\end{align}
and note that the left hand side does not depend on $m$ and substituting the bound in \eqref{eq:metric_bound} reveals the first term on the right hand side of \eqref{eq:iterated_inequality} converges to zero as $m\rightarrow\infty$
\begin{align*}
	\left(\frac{d}{2\tau(\theta_{n})}\right)^{m}\log\phi_{\theta_{n}}\left( 2^{-m}\varepsilon\right) 
	& \leq \left(\frac{d}{2\tau(\theta_{n})}\right)^{m}\log\left(C(\theta_{n},\gamma)2^{m\gamma} \varepsilon^{-\gamma}\right) \\
	& = \left(\frac{d}{2\tau(\theta_{n})}\right)^{m}\left(m\gamma\log 2 + \log(C(\theta_{n},\gamma)\varepsilon^{-\gamma})\right)  \xrightarrow{m\rightarrow\infty}0.
\end{align*}

So taking the limit of $m$ in \eqref{eq:iterated_inequality} gives
\begin{align*}
\log\phi_{\theta_{n}}(\varepsilon) & \leq \sum_{j=0}^{\infty}\left(\frac{d}{2\tau(\theta_{n})}\right)^{j}\log\chi_{n}\left( 2^{-j}\varepsilon\right) \\
& = \frac{2\tau(\theta_{n})}{(2\tau(\theta_{n})-d)}\log\chi_{n}(\varepsilon) + \sum_{j=0}^{\infty}\left(\frac{d}{2\tau(\theta_{n})}\right)^{j}\log\left(\frac{\chi_{n}(2^{-j}\varepsilon)}{\chi_{n}(\varepsilon)}\right) \\
& \leq \frac{2\tau(\theta_{n})}{(2\tau(\theta_{n})-d)}\log\chi_{n}(\varepsilon) + \log(2)\frac{d}{\tau_{k}^{-}}\sum_{j=0}^{\infty}\left(\frac{d}{2\tau_{k}^{-}}\right)^{j}j,
\end{align*}
and sum has a closed form which we can upper bound
\begin{align*}
	\log\phi_{\theta_{n}}(\varepsilon) \leq \frac{2\tau(\theta_{n})}{(2\tau(\theta_{n})-d)} \log\chi_{n}(\varepsilon) + \log(2)\left(\frac{d}{\tau_{k}^{-}}\right)\left(\frac{d}{2\tau_{k}^{-}}\right)\left(\frac{d}{2\tau_{k}^{-}}-1\right)^{-2}.
\end{align*}

Finally, exponentiating tells us that $\forall\varepsilon < \varepsilon^{*}$:
\begin{align*}
	\phi_{\theta_{n}}(\varepsilon) \leq C^{*}\varepsilon^{-2d/(2\tau(\theta_{n})-d)} \leq C^{*}\varepsilon^{-2d/(2\tau_{k}^{-}-d)},
\end{align*}
where we have collected the dependencies on $\Theta_{N}^{*}$ into $C^{*}$. In summary the following lemma which is analogous to the result by \citet[Lemma 3]{VanderVaart2011}, but with possibly changing hyperparameters, has been proved.
\begin{lemma}\label{lem:small_ball}
	Fix $N\in\bbN$ and suppose Assumptions 1-\ref{ass:target_extension} hold. Then, for $\varepsilon < \varepsilon^{*}$, where $\varepsilon^{*}$ is from \eqref{eq:ep_star}, and $n\geq N$: $-\log\Pi_{\theta_{n}}\left(\norm{f}_{L^{\infty}\left(\overline{\X}\right)} \leq \varepsilon \right) \leq C^{*}\varepsilon^{-2d/ (2\tau_{k}^{-}-d)}$.
\end{lemma}

\begin{corollary}\label{cor:gen_inv_bound}
	Fix $N\in\bbN$ and suppose Assumptions 1-\ref{ass:target_extension} hold. \sloppy Then, $\exists C = C\left(\norm{f}_{W^{\tau_{f}}_{2}(\X)},\Theta_{N}^{*}\right)$ such that for $n\geq N$, $\psi_{\theta_{n},f}^{-1}(n)\leq C\max\left(n^{-\tau_{f}/2\tau_{k}^{+}}, n^{d/4\tau_{k}^{-}-1/2}\right)$.
\end{corollary}
\begin{proof}
	By Lemma \ref{lem:decentering} and Lemma \ref{lem:small_ball}, using the restriction of $f^{\circ}$ to the closure of $\X$, $\exists C_{1} > 0$ such that $\forall\: n\geq N$ and $\varepsilon < \varepsilon^{*}$, where $\varepsilon^{*}$ is from \eqref{eq:ep_star}
	\begin{align*}
    	\phi_{\theta_{n},f}(\varepsilon)\varepsilon^{-2} 
    	& \leq C_{1}\left(\varepsilon^{-(2d/(2\tau(\theta_{n}) - d)) -2} + \varepsilon^{-2((\tau(\theta_{n}) - \tau)/\tau) -2}\right) \\
    	& \leq C_{1} \varepsilon^{-\frac{2\tau(\theta_{n})}{\min{(\tau,\tau(\theta_{n}) - (d/2))}}} 
    	 \leq C_{1}\left(\varepsilon^{-2\tau_{N}^{+}/\tau}\vee\varepsilon^{\left(d/(4\tau_{N}^{-})-1/2\right)^{-1}}\right).
	\end{align*}
	Set $\varepsilon_{n} = n^{-\left(\tau/2\tau_{k}^{+}\right)\wedge(d/4\tau_{k}^{-}-1/2)}$ then we know once $n$ is large enough that we have $\varepsilon_{n} < \varepsilon^{*}$ therefore $\exists \:C_{2}$ such that $\forall n\geq N$ $\phi_{\theta_{n},f}(\varepsilon_{n})\varepsilon_{n}^{-2}\leq C_{2}n$. Multiplying $\varepsilon_{n}$ by a constant to remove the factor of $C_{2}$ in the previous expression completes the proof.
\end{proof}

\section{Proof of Theorem \ref{thm:BO}}
The proof follows the proof by \citep[Theorem 1]{Bull2011}. The point $x_{n}$ satisfies
\begin{align*}
	f(x^{*})-f(x_{n})\leq f(x^{*})-R_{f}(x^{*}) - f(x_{n}) + R_{f}(x_{n})\leq 2\bignorm{f-R_{f}}_{L^{\infty}(\X)},
\end{align*} 
since $R_{f}(x_{n}) \geq R_{f}(x^{*})$ since $x_{n}$ was chosen as the maximizer of $R_{f}$. The points picked by the $(\gamma,F,n)$ strategy are quasi-uniform by \citet[Theorem 14, Theorem 18]{Wenzel2019} therefore taking $n_{0}$ to be large enough to ensure that the fill distance obtained from the strategy is small enough to employ Theorem \ref{thm:Misspecified_Sobolev_Interpolation} completes the proof.

\vskip 0.2in

\bibliography{gp_rates.bib}

\end{document}